\documentclass[11pt]{amsart}

\usepackage{parskip}
\allowdisplaybreaks

\usepackage{geometry} 
\usepackage{framed} 

\usepackage{tikz}
\tikzset{every picture/.style={line width=1pt}} 

\usepackage{setspace}
\setdisplayskipstretch{2.5}

\usepackage{amsthm} 
\usepackage{amsmath}
\usepackage{amssymb}

\usepackage{hyperref}

\usepackage{amsfonts} 
\newcommand\A{\mathcal{A}}
\newcommand\B{\mathcal{B}}
\newcommand\FB{\mathfrak{B}}

\newcommand\CE{\mathcal{E}}
\newcommand\CF{\mathcal{F}}
\newcommand\G{\mathcal{G}}
\renewcommand\H{\mathcal{H}}
\newcommand\CK{\mathcal{K}}
\newcommand\CL{\mathcal{L}}

\newcommand\N{\mathbb{N}}
\newcommand\CP{\mathcal{P}}
\newcommand\FP{\mathfrak{P}}
\newcommand\R{\mathbb{R}}

\newcommand\Z{\mathbb{Z}}

\newcommand\w{\omega}

\usepackage{stmaryrd} 

\newcommand\id{\textnormal{id}}
\newcommand\diam{\textnormal{diam}}

\newcommand\Lip{\textnormal{Lip}}
\newcommand\dist{\textnormal{dist}}

\newcommand\ext{\mathrm{d}}
\newcommand\del{\partial}

\newcommand{\vertiii}[1]{{\left\vert\kern-0.25ex\left\vert\kern-0.25ex\left\vert #1 \right\vert\kern-0.25ex\right\vert\kern-0.25ex\right\vert}} 

\newtheoremstyle{newtheoremstyle}
{3pt}
{3pt}
{\itshape}
{\parindent}
{\bfseries}
{.}
{0.5em}
{} 

\newtheoremstyle{newtheoremstyledefn}
{3pt}
{3pt}
{}
{\parindent}
{\bfseries}
{.}
{0.5em}
{}

\theoremstyle{newtheoremstyle}
\newtheorem{theorem}{Theorem}
\newtheorem{lemma}{Lemma}

\newtheorem{corollary}{Corollary}

\theoremstyle{newtheoremstyledefn}
\newtheorem{defn}{Definition}

\numberwithin{equation}{section} 
\numberwithin{theorem}{section}
\numberwithin{lemma}{section}
\numberwithin{prop}{section}
\numberwithin{corollary}{section}
\numberwithin{defn}{section}

\usepackage{fancyhdr}
\begin{document}

\title{A Campanato Regularity Theory for Multi-Valued Functions with Applications to Minimal Surface Regularity Theory}

\author{Paul Minter}
\address{Department of Pure Mathematics and Mathematical Statistics, University of Cambridge}
\email{pdtwm2@cam.ac.uk}

\begin{abstract}
	The regularity theory of the Campanato space $\CL^{(q,\lambda)}_k(\Omega)$ has found many applications within the regularity theory of solutions to various geometric variational problems. Here we extend this theory from single-valued functions to multi-valued functions, adapting for the most part Campanato's original ideas (\cite{campanato}). We also give an application of this theory within the regularity theory of stationary integral varifolds. More precisely, we prove a regularity theorem for certain \textit{blow-up classes} of multi-valued functions, which typically arise when studying blow-ups of sequences of stationary integral varifolds converging to higher multiplicity planes or unions of half-planes. In such a setting, based in part on ideas in \cite{wickstable}, \cite{minterwick}, and \cite{beckerwick}, we are able to deduce a boundary regularity theory for multi-valued harmonic functions; such a boundary regularity result would appear to be the first of its kind for the multi-valued setting. In conjunction with \cite{minter}, the results presented here establish a regularity theorem for stable codimension one stationary integral varifolds near classical cones of density $\frac{5}{2}$.
\end{abstract}

\maketitle

\tableofcontents

Multi-valued functions were first introduced by F.~Almgren (\cite{almgren}) to study the branching behaviour of area-minimising currents. In recent years, multi-valued functions have successfully been used to describe the structure of stationary integral varifolds near multiplicity two planes, and in particular near multiplicity two branch points (see \cite{wickramasekera2008regularity}, \cite{minterwick}, \cite{beckerwick}). In order to understand the structure of stationary integral varifolds close to higher multiplicity planes (\cite{minterwick}) and higher multiplicity non-flat cones (\cite{minter}), one needs to develop suitable regularity results for special classes of multi-valued functions, known as (\textit{proper}) \textit{blow-up classes} (see \cite{wickstable} for a key example). The functions within a proper blow-up class are generated by taking scaling limits of approximate graphical representations of sequences of certain stationary integral varifolds converging to a fixed stationary cone. The functions within a given blow-up class are typically multi-valued as opposed to single-valued, and in most cases of interest defined either on an open ball or half-ball in a Euclidean plane; moreover, the functions in the blow-up class inherit certain integral estimates from the stationarity assumption on the varifolds. When it is possible to use these estimates to deduce $C^{1,\alpha}$ regularity (or even generalised-$C^{1,\alpha}$ regularity -- see \cite{minterwick}) of the functions within a blow-up class is a key problem within geometric measure theory, and is the motivation for our work here. In conjunction with \cite{minter}, the work here establishes a $C^{1,\alpha}$ regularity theory for stable codimension one stationary integral varifolds which are close to a stationary integral classical cone of vertex density $\frac{5}{2}$ (see \cite{minter} for explanations of this terminology); this is the first instance of a regularity theorem in a non-flat setting of higher multiplicity when branch points may be present in the nearby varifold.

The present paper is divided into two parts. In Sections \ref{sec:prelim} and \ref{sec:campanato} we define Campanato spaces for multi-valued functions and develop their regularity theory, mirroring that seen in Campanato's original work on the single-valued Campanato spaces (\cite{campanato}). Other than some technical changes, Campanato's main ideas extend readily to this setting. We anticipate that such a result will find more applications than just to regularity theory of minimal submanifolds, which is why we choose to present it separately; for our purposes the results serve as a black box to apply to the minimal submanifold setting. We also give some adaptations of the Campanato regularity theorem suited to the minimal submanifold setting. In Section \ref{sec:minimal_surface} we introduce our notion of a (proper) blow-up class over a \textit{half-plane}. Using the results from Section \ref{sec:campanato} and adapting ideas seen in \cite{wickstable}, \cite{minterwick}, and \cite{beckerwick}, we deduce that the functions within such a class are in fact $C^{1,\alpha}$, with $C^{1,\alpha}$ extensions up to the boundary of the half-plane, for some $\alpha\in (0,1)$ only depending on the blow-up class. We note that, in particular, as the functions within these blow-up classes are multi-valued $C^{1,\alpha}$ harmonic in the interior, these results give the first instance of a $C^{1,\alpha}$ boundary regularity theory for multi-valued harmonic functions.

\textbf{Acknowledgements:} This work was supported by the UK Engineering and Physical Sciences Research Council (EPSRC) grant EP/L016516/1 for the University of Cambridge Centre for Doctoral Training, the Cambridge Centre for Analysis.

\section{Preliminaries}\label{sec:prelim}

The standard references for multi-valued functions are \cite{almgren}, \cite{de2010almgren}, and \cite{de2013multiple}; the reader wishing to attain a broader background on multi-valued functions is recommended to consult these, as we shall only briefly recall the notions we need.

For $T\in \R^m$, we write $\llbracket T\rrbracket$ for the Dirac mass centred at $T$. For $Q\in \{1,2,\dotsc\}$ we write $\A_Q(\R^m)$ for the space of \textit{unordered $Q$-tuples}, i.e.
$$\A_Q(\R^m):= \left\{\sum^Q_{i=1}\llbracket T_i\rrbracket\,:\, T_i\in \R^m\text{ for }i=1,\dotsc,Q\right\}.$$
We equip $\A_Q(\R^m)$ with the metric $\G$ defined by
$$\G\left(\sum_i\llbracket R_i\rrbracket\ ,\ \sum_i\llbracket T_i\rrbracket\right):= \min_{\sigma\in S_Q}\left(\sum_i |R_i-T_{\sigma(i)}|^2\right)^{1/2}$$
where $S_Q$ is the group of permutations of $\{1,\dotsc,Q\}$. It is easy to check that $(\A_Q(\R^m),\G)$ is a complete metric space. For $T\in \A_Q(\R^m)$ we also write $|T|:= \G(T,Q\llbracket 0\rrbracket)$. We stress that $\A_Q(\R^m)$ is \textit{not} a vector space, as there is no natural notion of addition for unordered $Q$-tuples.

Let $\Omega\subset\R^n$ be a domain. A $Q$\textit{-valued function} is a function $u:\Omega\to \A_Q(\R^m)$; we write $u(x) = \sum_i\llbracket u_i(x)\rrbracket$ for some $u_i(x)\in \R^m$ which are unique up to permutations. In the case $m=1$, we can define single-valued functions $\tilde{u}_i:\Omega\to \R$ with $\tilde{u}_1\geq \tilde{u}_2\geq\cdots \geq \tilde{u}_Q$ and $u(x) = \sum_i \llbracket \tilde{u}_i(x)\rrbracket$ for all $x$. Even though $\A_Q(\R^m)$ is not a vector space, we shall abuse notation and write for a single-valued function $f:\Omega\to \R^m$ and $Q$-valued function $g:\Omega\to \A_Q(\R^m)$ the function $f+g$ to mean the $Q$-valued function given by $x\mapsto \sum_i\llbracket f(x) + g_i(x)\rrbracket$.

For $\alpha\in (0,1]$, we define the space of $Q$\textit{-valued} $\alpha$\textit{-H\"older continuous functions}, which we denote by $C^{0,\alpha}(\Omega;\A_Q(\R^m))$, in the usual way for functions between metric spaces. Similarly we define $C^{0,\alpha}(\overline{\Omega};\A_Q(\R^m))$ to be those functions in $C^{0,\alpha}(\Omega;\A_Q(\R^m))$ which have $\alpha$-H\"older continuous extensions to $\overline{\Omega}$. For each $p\in [1,\infty)$ we define $L^p(\Omega;\A_Q(\R^m))$ for the $Q$-valued functions which have $\|u\|_{p} := \|\G(u,Q\llbracket 0\rrbracket)\|_{L^p(\Omega)}<\infty$.

\begin{defn}
	We say that a $Q$-valued function $u:\Omega\to \A_Q(\R^m)$ is \textit{differentiable} at $x_0\in \Omega$ if there exist $Q$ matrices $A_i\in \R^{m\times n}$, $i=1,\dotsc,Q$, which satisfy
	$$\lim_{h\to0}\frac{\G(u(x_0+h),L(h))}{|h|} = 0$$
	where $L(h):= \sum_i\llbracket u_i(x_0) + A_i(h)\rrbracket$; we then write $Du_i(x_0) := A_i$, and $Du(x_0)$ is the $Q$-valued function given by $\sum_{i}\llbracket A_i\rrbracket :\R^n\to \A_Q(\R^{m})$.
\end{defn}

We can then define spaces such as $C^k(\Omega;\A_Q(\R^m))$, $C^{k,\alpha}(\Omega;\A_Q(\R^m))$, $k=1,2,\dotsc$, in the natural ways, and we denote the corresponding \lq\lq norms'' and \lq\lq semi-norms'' by $|u|_{k,\alpha;\Omega}$, $[u]_{k,\alpha;\Omega}$, respectively. Note that as there is no vector space structure, these are not norms in the usual sense, however we shall still refer to them as norms and semi-norms.

One result which we need later is the following version of the Lebesgue differentiation theorem for multi-valued functions (here, $\H^n$ denotes the $n$-dimensional Hausdorff measure).

\begin{theorem}[Multi-Valued Lebesgue Differentiation Theorem]\label{thm:lebesgue_differentiation}
	Suppose $\Omega\subset\R^n$ is a domain, $q\in [1,\infty)$, and $u\in L^q(\Omega;\A_Q(\R^m))$. Then for $\H^n$-a.e. $x_0\in \Omega$ we have
	$$\lim_{\rho\to 0}\frac{1}{\w_n\rho^n}\int_{B_\rho(x_0)}\G(u(x),u(x_0))^q\ \ext x = 0$$
	where here $\w_n := \H^n(B_1(0))$ is the volume of the $n$-dimensional unit ball in $\R^n$.
\end{theorem}

\begin{proof}
	From \cite{almgren} (see also \cite[Theorem 2.1]{de2010almgren}) we know that there exists $N = N(m,Q)$ and a bi-Lipschitz function $\xi:\A_Q(\R^m)\to \xi(\A_Q(\R^m))\subset \R^N$ such that $\Lip(\xi)\leq 1$ and $\Lip(\xi^{-1})\leq C(m,Q)$. Now consider the function $\tilde{u}:= \xi\circ u:\Omega\to \R^N$. Then since $u\in L^q(\Omega;\A_Q(\R^m))$ and $\xi$ is Lipschitz we see that $\tilde{u}\in L^q_{\text{loc}}(\Omega)$; indeed, fixing some $p\in \Omega$ we have for all $x\in \Omega$,
	$$|\tilde{u}(x)| \leq |\tilde{u}(p)| + \G(u(x),u(p)) \leq |\tilde{u}(p)| + |u(x)| + |u(p)|.$$
	Hence applying the usual Lebesgue differentiation theorem to $\tilde{u}$ we see that for $\H^n$-a.e. $x_0\in \Omega$,
	$$\lim_{\rho\to 0}\frac{1}{\w_n\rho^n}\int_{B_\rho(x_0)}|\tilde{u}(x)-\tilde{u}(x_0)|^q\ \ext x =0.$$
	Hence for $\H^n$-a.e. $x_0\in \Omega$,
	\begin{align*}
		\frac{1}{\w_n\rho^n}\int_{B_\rho(x_0)}\G(u(x),u(x_0))^q & = \frac{1}{\w_n\rho^n}\int_{B_\rho(x_0)}\G(\xi^{-1}(\tilde{u}(x)), \xi^{-1}(\tilde{u}(x_0)))^q\ \ext x\\
		& \leq \frac{C(m,Q)^q}{\w_n\rho^n}\int_{B_\rho(x_0)}|\tilde{u}(x)-\tilde{u}(x_0)|^q\ \ext x\\
		& \to 0
	\end{align*}
	and so we are done.
\end{proof}
\begin{defn}
	We say that $p:\R^n\to \A_Q(\R^m)$ is a $Q$\textit{-valued polynomial} if there exist $Q$ functions $p_i:\R^n\to \R^m$ such that for each $i=1,\dotsc,Q$ and $j=1,\dotsc,m$, $p_i^j:\R^n\to \R$ is a polynomial function, and for all $x\in \R^n$:
	$$p(x) = \sum_i \llbracket p_i(x)\rrbracket.$$
	We define the \textit{degree} of $p$ by $\deg(p):= \max_{i,j}\deg(p^j_i)$.
\end{defn}

For each $k=0,1,\dotsc,$ we write $\CP_k$ for the set of all $Q$-valued polynomials $p:\R^n\to\A_Q(\R^m)$ with $\deg(p)\leq k$.

\begin{defn}
	Let $q\in [1,\infty)$, $\lambda\in (0,\infty)$, $k\in \{0,1,2,\dotsc\}$, and $\Omega\subset\R^n$ be a domain. We then define the $Q$\textit{-valued Campanato space} $\CL^{q,\lambda}_k(\Omega;\A_Q(\R^m))$ to be the set of functions $u\in L^q(\Omega;\A_Q(\R^m))$ which obey
	$$\vertiii{u}_{k,q,\lambda;\Omega}:= \sup_{\substack{x_0\,\in\, \overline{\Omega}, \\ \rho\,\in (0,\diam(\Omega)]}}\left[\rho^{-\lambda}\inf_{P\in \CP_k}\int_{\Omega\cap B_\rho(x_0)}\G(u(x),P(x))^q\ \ext x\right]^{1/q}<\infty.$$
\end{defn}

\textbf{Remark:} We will always suppress the domain dependence on our norms and semi-norms when the domain is clear from context, and again we still refer to them as ``norms'' and ``semi-norms'' despite $\A_Q(\R^m)$ not being a vector space.

Of course, $\vertiii{u}_{k,q,\lambda} =0$ whenever $u\in \CP_k$, and so $\vertiii{\cdot}_{k,q,\lambda}$ is only a semi-norm. To make a norm on $\CL^{q,\lambda}_k(\Omega;\A_Q(\R^m))$ we instead work with
$$\|u\|_{k,q,\lambda}^q := \|u\|_{q}^q + \vertiii{u}_{k,q,\lambda}^q.$$

In the same way as in \cite{campanato}, for the regularity theory of $\CL^{q,\lambda}_k(\Omega;\A_Q(\R^m))$ we will restrict ourselves to domains $\Omega\subset\R^n$ which obey a certain mass condition:

\begin{defn}
	Let $A>0$. We say a bounded domain $\Omega\subset\R^n$ is $A$\textit{-weighted} if
	$$\H^n(\Omega\cap B_\rho(x)) \geq A\rho^n\ \ \ \ \text{ for all }x\in \overline{\Omega} \text{ and }\rho\in (0,\diam(\Omega)]$$
	where $\diam(\Omega):= \sup_{x,y\in \Omega}|x-y|$ is the diameter of $\Omega$. 
\end{defn}
Thus, an $A$-weighted domain always take up a fixed proportion of every ball centred at points on $\overline{\Omega}$. Key examples of $A$-weighted domains (for some fixed $A>0$) include the open ball $B_1(0)\subset \R^n$ and half-ball $B_1(0)\cap \{x^1>0\}\subset\R^n$. We note that $A$-weighted domains are referred to as \textit{type (A)} domains in \cite{rafeiro2013morrey} and \textit{type (I)} domains in \cite{campanato}.

\section{Multi-Valued Campanato Regularity Theory}\label{sec:campanato}

In this section we develop the regularity theory of $\CL^{q,\lambda}_k(\Omega;\A_Q(\R^m))$, following for the most part the ideas seen in \cite{campanato}. Afterward, we provide extensions relevant to geometric problems, such as the regularity theory of stationary integral varifolds. The main general regularity result of this section is the following theorem:

\begin{theorem}[Regularity of $\CL^{q,\lambda}_k(\Omega;\A_Q(\R^m))$]\label{thm:campanato}
	Suppose $\Omega\subset\R^n$ is a convex $A$-weighted domain and $u\in \CL^{q,\lambda}_k(\Omega;\A_Q(\R^m))$. Suppose that $\lambda$ obeys $n+\ell q < \lambda < n + (\ell +1)q$ for some $\ell \in \{0,1,\dotsc,k\}$. Then $u\in C^{\ell,\alpha}(\overline{\Omega};\A_Q(\R^m))$, with the estimate
	$$[u]_{\ell,\alpha} \leq C\vertiii{u}_{\ell,q,\lambda}$$
	where $\alpha:= \frac{\lambda - n-\ell q}{q}$ and $C = C(n,m,k,\ell,q,Q,A,\lambda)$. Moreover, $\CL^{q,\lambda}_k(\Omega;\A_Q(\R^m))$ is continuously isomorphic to $\CL^{q,\lambda}_\ell(\Omega;\A_Q(\R^m))$.
\end{theorem}

Let us fix $u\in \CL^{q,\lambda}_k(\Omega;\A_Q(\R^m))$ throughout. First note that, for each $x_0\in\overline{\Omega}$ and $\rho\in (0,\diam(\Omega)]$, we can find a $Q$-valued polynomial $\tilde{P}\in \CP_k$ such that
$$\inf_{P\in \CP_k}\int_{\Omega\cap B_\rho(x_0)} \G(u,P)^q = \int_{\Omega\cap B_\rho(x_0)}\G(u,\tilde{P})^q.$$
Indeed, each $P\in \CP_k$ is determined by polynomials $p^{i,j}:\R^n\to \R$, which are in turn determined by coefficients $a = (a^{i,j}_p)_{i,j,p}$ where $i=1,\dotsc,Q$, $j=1,\dotsc,m$, and $p\in \N^n$ is a multi-index with $|p|\leq k$. Then the function $f:\R^M\to \R$, $M = M(n,m,k,Q)$, sending
$$f:((a^{i,j}_p)_{i,j,p})\longmapsto \int_{\Omega\cap B_{\rho}(x_0)}\G(u,P)^q$$
where $P$ is the polynomial generated by $a = (a^{i,j}_p)_{i,j,p}$ as above, is a continuous function. Therefore to see the infimum is attained one just needs to show that $f(a)\to \infty$ as $|a|\to \infty$. But the triangle inequality gives
$$f(a):= \int_{\Omega\cap B_\rho(x_0)} \G(u,P)^q \geq \left|\|u\|_{L^q(\Omega\cap B_\rho(x_0))} - \|P\|_{L^q(\Omega\cap B_\rho(x_0))}\right|^q$$
which clearly $\to \infty$ as $|a|\to \infty$. Write $P_{x_0,\rho}:= \tilde{P}$ for a choice of $Q$-valued polynomial $\tilde{P}$ attaining the infimum.

We also write $a^{i,j}_p(x_0,\rho) := D^{p}P^{i,j}_{x_0,\rho}(x_0)$, so that 
$$P^{i,j}_{x_0,\rho}(x) = \sum_{|p|\leq k}\frac{a_p^{i,j}(x_0,\rho)}{p!}(x-x_0)^p.$$
We first prove a crucial integral estimate for comparing two $Q$-valued polynomials. Whilst the estimate is similar to that seen in \cite[Lemma 2.I]{campanato}, we need to be more careful due to the lack of any vector space structure. Indeed, when $Q>1$, $\A_Q(\R^{n\times m}) \not\cong \A_Q(\R^n)\times \A_Q(\R^m)$, since in the former space the $(n\times m)$-tuples need to be close to one another for a given permutation, whilst in the later we are able to use different permutations. For us, this would correspond to being able to permute the constants in the $Q$-valued polynomials between the polynomials independently of, say, the linear coefficients, which will not be allowed.

\begin{lemma}\label{lemma:poly_estimate}
	Let $q\in [1,\infty)$ and $E\subset B_\rho(x_0)$ be a measurable subset with $\H^n(E)\geq A\rho^n$ for some $A>0$. Then for any pair of $Q$-valued polynomials $F,G\in \CP_k$, given by
	$$F^{i,j}(x):= \sum_{|p|\leq k}\frac{a^{i,j}_p}{p!}(x-x_0)^p\ \ \ \ \text{and}\ \ \ \ G^{i,j}(x):= \sum_{|p|\leq k}\frac{b^{i,j}_p}{p!}(x-x_0)^p$$
	we have, defining $a_\rho\in \A_Q(\R^M)$ and $b_\rho\in \A_Q(\R^M)$, $M = M(n,m,k)$, by $a_\rho^i := (\rho^{|p|}a^{i,j}_p)_{j,p}$ and $b_\rho^i := (\rho^{|p|}b^{i,j}_p)_{j,p}$,
	$$\G(a_\rho,b_\rho)^q \leq \frac{C_1}{\rho^n}\int_E\G(F(x),G(x))^q\ \ext x$$
	for some $C_1 = C_1(n,m,k,q,Q,A)$, which is in particular independent of $F,G,E$.
\end{lemma}

\begin{proof}
	First note that if we set $\tilde{F}(x):= F(x_0+\rho x)$ and $\tilde{G}(x):= F(x_0+\rho x)$, then we see
	$$\tilde{F}^{i,j}(x) = \sum_{|p|\leq k}\frac{(\rho^{|p|}a^{i,j}_p)}{p!}x^p\ \ \ \ \text{and}\ \ \ \ \tilde{G}^{i,j}(x) = \sum_{|p|\leq k}\frac{(\rho^{|p|}b^{i,j}_p)}{p!}x^p$$
	and so we see that it suffices to prove the case $x_0 = 0$ and $\rho = 1$. So let us restrict to this case.
	
	Thus, suppose we have two $Q$-valued polynomials $F,G$, and write
	$$F = \sum^Q_{i=1}\llbracket F^i\rrbracket\ \ \ \ \text{and}\ \ \ \ G = \sum^Q_{i=1}\llbracket G^i\rrbracket$$
	where $F^i = (F^{i,j})_{j=1}^m$, $G^i = (G^{i,j})_{j=1}^m$, and
	$$F^{i,j}(x) = \sum_{|p|\leq k}a^{i,j}_p x^p\ \ \ \ \text{and}\ \ \ \ G^{i,j}(x) = \sum_{|p|\leq k}b^{i,j}_p x^p.$$
	Note that we have combined with the coefficients the $p!$ factors for notational simplicity, and these can simply be removed at the end by absorbing them into the constant. Define $a,b\in \A_Q(\R^M)$, where $M = M(n,m,k)$, where $a^i = (a^{i,j}_p)_{j,p}$ and $b^i = (b^{i,j}_p)_{j,p}$, i.e., $a^i$ determines all the polynomials in $F^i$, etc. Now define:
	$$\alpha_{F,G}:= \G(a,b) \equiv \inf_{\sigma}\left(\sum^Q_{i=1}\sum_{j=1}^m\sum_{|p|\leq k}|a^{i,j}_p - b^{\sigma(i),j}_p|^2\right)^{1/2}.$$
	Let $\CF$ denote all pairs $(F,G)$ of $Q$-valued polynomials in $\CP_k$ which have $\alpha_{F,G} = 1$. Let $\CE$ be the class of subsets $E\subset B_1(0)$ which have $\H^n(E)\geq A$. Then define:
	$$\gamma := \inf_{\substack{(F,G)\in \CF \\ E\in \CE}}\int_E\G(F,G)^q.$$
	We claim that $\gamma\geq C$ for some constant $C = C(n,m,k,q,Q,A)>0$. Indeed, fix $(F,G)\in \CF$, and using the same notation as just defined, for each $\sigma\in S_Q$, $i=1,\dotsc,Q$, $j=1,\dotsc,m$, and $|p|\leq k$, set
	$$c^{\sigma}_{i,j,p}:= a^{i,j}_p - b^{\sigma(i),j}_p\ \ \ \ \text{and}\ \ \ \ P^\sigma_{i,j}(x):= \sum_{|p|\leq k}c^{\sigma}_{i,j,p}x^p.$$
	Then we have for each $x\in B_1(0)$,
	$$\G(F(x),G(x))^2 := \inf_{\sigma}\sum^Q_{i=1}|F^i(x)-G^i(x)|^2 = \inf_{\sigma}\sum^Q_{i=1}\sum_{j=1}^m\left|\sum_{|p|\leq k}(a^{i,j}_p - b^{\sigma(i),j}_p)x^p\right|^2$$
	i.e.,
	$$\G(F(x),G(x))^2 = \inf_\sigma \sum_{i,j}|P^\sigma_{i,j}(x)|^2.$$
	Now define for each $\sigma\in S_Q$, $E_\sigma\subset E$ to be the set of points $x\in E$ for which this infimum equals $\sum_{i,j}|P^\sigma_{i,j}(x)|^2$. Clearly $E = \cup_\sigma E_\sigma$, and so we must be able to find $\sigma'\in S_Q$ for which $\H^n(E_{\sigma'}) \geq A/|S_Q| = A/(Q!)$. Now note that
	$$\int_E \G(F,Q)^q \geq \int_{E_{\sigma'}}\G(F,G)^q = \int_{E_{\sigma'}}\sum_{i,j}|P^{\sigma'}_{i,j}|^2.$$
	Now fix each $\sigma\in S_Q$ we have by definition of $\alpha_{F,G}$ and since $(F,G)\in \CF$,
	$$\sum^Q_{i=1}\sum_{j=1}^m\sum_{|p|\leq k}|c^{\sigma}_{i,j,p}|^2 \geq \alpha^2_{F,G} = 1$$
	and hence applying this with $\sigma = \sigma'$, we must be able to find some $i'\in \{1,\dotsc,Q\}$, $j'\in \{1,\dotsc,m\}$, and $|p'|\leq k$ for which
	$$|c^{\sigma'}_{i',j',p'}| \geq \frac{1}{mQN(n,k)}.$$
	Now applying \cite[Lemma 2.I]{campanato} with $A/(Q!)$ in place of $A$, $F$ in place of $E$, where $F\subset E_{\sigma'}$ has $\H^n(F) = A/(Q!)$, $p = p'$, we have:
	$$\int_{E_{\sigma'}}\sum_{i,j}|P^{\sigma'}_{i,j}|^q \geq \int_F |P^{\sigma'}_{i',j'}|^q \geq \tilde{C}|D^{p'}P^{\sigma'}_{i',j'}(0)|^q = \tilde{C}|(p_{\sigma'}!)\cdot c^{\sigma'}_{i',j',p'}|^q \geq C_*$$
	where $\tilde{C} = \tilde{C}(n,k,q,A/(Q!))>0$ and $C_* = C_*(n,m,k,q,Q,A)>0$. Hence we see that
	$$\int_E \G(F,Q)^q \geq C_*$$
	and so as $(F,G)\in \CF$ and $E\in \CE$ were arbitrary, we see $\gamma\geq C_*$, as desired.
	
	Now given any arbitrary pair of distinct $Q$-valued functions $(F,G)$, note that $(F/\alpha_{F,G},G/\alpha_{F,G})\in \CF$ (if $\alpha_{F,G}=0$ the result is trivial, so assume $\alpha_{F,G}>0$), and so for any $E\subset A$ with $\H^n(E)\geq A$ we have
	$$\int_E \G(F,G)^q \geq C_*\alpha^q_{F,G}$$
	which completes the proof.
\end{proof}

\textbf{Remark:} From Lemma \ref{lemma:poly_estimate} we see that, writing $a_{(r)}^i:= (a^{i,j}_p)_{j=1,\dotsc,m, |p|\leq r}$ and $b_{(r)}^i := (b^{i,j}_p)_{j=1,\dotsc,m, |p|\leq r}$, one has:
$$\G(a_{(r)},b_{(r)})^q\equiv \inf_\sigma\left(\sum_{i=1}^Q\sum^m_{j=1}\sum_{|p|\leq r}|a^{i,j}_p-b^{\sigma(i),j}_p|^2\right)^{q/2} \leq \frac{C_1}{\min\{1,\rho\}^{n+rq}}\int_E \G(F,G)^q\ \ext x.$$

\begin{lemma}\label{estimate_1}
	Let $u\in \CL^{q,\lambda}_k(\Omega;\A_Q(\R^m))$. Then for each $x_0\in \overline{\Omega}$, $\rho\in (0,\diam(\Omega)]$, and $\ell \in \{0,1,2,\dotsc\}$, for some $C = C(q,\lambda)$ we have:
	$$\int_{\Omega\cap B_{2^{-\ell-1}\rho}(x_0)}\G(P_{x_0,2^{-\ell}\rho},P_{x_0,2^{-\ell-1}\rho})^q \leq C2^{-\ell\lambda}\rho^\lambda\cdot\vertiii{u}^q_{k,q,\lambda}.$$
\end{lemma}

\begin{proof}
	By the triangle inequality for $\G$ we simply have:
	\begin{align*}
		\int_{\Omega\cap B_{2^{-\ell-1}\rho}(x_0)}&\G(P_{x_0,2^{-\ell}\rho},P_{x_0,2^{-\ell-1}\rho})^q\\
		& \leq 2^q\int_{\Omega\cap B_{2^{-\ell}\rho}(x_0)}\G(P_{x_0,2^{-\ell}\rho},u)^q + 2^q\int_{\Omega\cap B_{2^{-\ell-1}\rho}(x_0)}\G(P_{x_0,2^{-\ell-1}\rho},u)^q\\
		& \leq 2^q\cdot (2^{-\ell}\rho)^\lambda\vertiii{u}^q_{k,q,\lambda} + 2^q\cdot (2^{-\ell-1}\rho)^\lambda\vertiii{u}^q_{k,q,\lambda}\\
		& = \left[2^q + 2^{q-\lambda}\right]\cdot 2^{-\ell\lambda}\rho^\lambda \vertiii{u}^q_{k,q,\lambda}.\hfill\qedhere
	\end{align*}
\end{proof}
The next lemma will be the first step towards a H\"older estimate for the $k$'th order derivatives and differentiability properties for lower order derivatives.
\begin{lemma}\label{estimate_2}
	Suppose $\Omega\subset\R^n$ is an $A$-weighted domain and $u\in \CL^{q,\lambda}_k(\Omega;\A_Q(\R^m))$. Then for each pair of points $x_0,y_0\in \overline{\Omega}$ with $\rho:= |x_0-y_0| \leq \frac{1}{2}\diam(\Omega)$, we have
	$$\G(a_{k}(x_0,2\rho),a_{k}(y_0,2\rho))^q \leq C_1 2^{q+1+\lambda}\vertiii{u}^q_{k,q,\lambda}\cdot \rho^{\lambda-n-kq}$$
	where $C_1 = C_1(q,\lambda)$ is as in Lemma \ref{lemma:poly_estimate} and $a_k^i(x_0,2\rho) = (a^i_p(x_0,2\rho))_{|p|=k}$, etc.
\end{lemma}
\begin{proof}
	Fix $x_0,y_0\in \overline{\Omega}$ as in the statement and set $\rho:=|x_0-y_0|$. Then by the triangle inequality for $\G$ and simple set inclusions one has
	\begin{align*}
		\int_{\Omega\cap B_\rho(x_0)}\G(P_{x_0,2\rho},P_{y_0,2\rho})^q & \leq 2^q\int_{\Omega\cap B_{2\rho}(x_0)}\G(P_{x_0,2\rho},u)^q + 2^q\int_{\Omega\cap B_{2\rho}(y_0)}\G(P_{y_0,2\rho},u)^q]\\
		& \leq 2^q\cdot (2\rho)^\lambda \vertiii{u}^q_{k,q,\lambda} + 2^q\cdot (2\rho)^\lambda \vertiii{u}^q_{k,q,\lambda}\\
		& = 2^{q+\lambda+1}\vertiii{u}^q_{k,q,\lambda}\cdot\rho^\lambda.
	\end{align*}
	Moreover from Lemma \ref{lemma:poly_estimate}, since $\Omega$ is $A$-weighted and because the $k$'th derivative of any polynomial of degree $k$ is constant (so in particular does not depend on the point we evaluate at, as Lemma \ref{lemma:poly_estimate} requires both polynomials to be centred at the same point) we have
	$$\G(a_{k}(x_0,2\rho),a_{k}(y_0,2\rho))^q \leq \frac{C_1}{\rho^{n+kq}}\int_{\Omega\cap B_\rho(x_0)}\G(P_{x_0,2\rho},P_{y_0,2\rho})^q$$
	since we can just look at the order $k$ terms. Combining these inequalities we get the result.
\end{proof}
\begin{lemma}\label{estimate_3}
	Suppose $\Omega\subset \R^n$ is an $A$-weighted domain and $u\in \CL^{q,\lambda}_k(\Omega;\A_Q(\R^m))$. Then there exists $C = C(n,m,k,q,\lambda,Q,A)$ such that, for all $x\in \overline{\Omega}$, $\rho\in (0,\min\{1,\diam(\Omega)\}]$, $i=0,1,2,\dotsc,$ and $r\leq k$ we have
	$$\G(a_{(r)}(x_0,\rho),a_{(r)}(x_0,2^{-i}\rho)) \leq C\vertiii{u}_{k,q,\lambda}\sum^{i-1}_{j=0}2^{j\left(\frac{n+rq-\lambda}{q}\right)}\cdot\rho^{\frac{\lambda-n-rq}{q}}.$$
\end{lemma}
\begin{proof}
	For any such $x_0,\rho,i,r,$ we have
	\begin{align*}
		\G(a_{(r)}(x_0,\rho),a_{(r)}(x_0,2^{-i}&\rho)) \leq \sum^{i-1}_{j=0}\G(a_{(r)}(x_0,2^{-j}\rho),a_{(r)}(x_0,2^{-j-1}\rho))\\
		& \leq \sum^{i-1}_{j=0}\left(\frac{C_1}{(2^{-j-1}\rho)^{n+rq}}\int_{\Omega\cap B_{2^{-j-1}\rho}(x_0)}\G(P_{x_0,2^{-j}\rho},P_{x_0,2^{-j-1}})^q\right)^{1/q}\\
		& \leq \sum^{i-1}_{j=0}\left(C_12^{(j+1)(n+rq)}\rho^{-n-rq}\cdot C\cdot 2^{-j\lambda}\rho^\lambda\vertiii{u}^q_{k,q,\lambda}\right)^{1/q}\\
		& \leq C\vertiii{u}_{k,q,\lambda}\sum^{i-1}_{j=0}2^{j\left(\frac{n+rq-\lambda}{q}\right)}\cdot\rho^{\frac{\lambda-n-rq}{q}}
	\end{align*}
	as desired, where the second inequality is Lemma \ref{lemma:poly_estimate} and the third inequality is Lemma \ref{estimate_1}.
\end{proof}

\begin{lemma}\label{estimate_4}
	Suppose $\Omega\subset\R^n$ is an $A$-weighted domain and $u\in \CL^{q,\lambda}_k(\Omega;\A_Q(\R^m))$, with $n+rq<\lambda \leq n+(r+1)q$ for some $r\in \{0,1,\dotsc,k\}$. Then for every $|p|\leq r$, there exists a function $v_p:\overline{\Omega}\to \A_Q(\R^m)$ such that for every $x_0\in \overline{\Omega}$ and $\rho\in (0,\min\{1,\diam(\Omega)\}]$ we have
	$$\G(a_{(r)}(x_0,\rho),v_{(r)}(x_0)) \leq C\vertiii{u}_{k,q,\lambda}\cdot\rho^{\frac{\lambda-n-rq}{q}}$$
	where $C = C(n,m,k,q,\lambda,Q,A)$. In particular $a_{(r)}(\cdot,\rho)\to v_{(r)}(\cdot)$ uniformly as $\rho\to 0$.
\end{lemma}

\begin{proof}
	Fix $x_0,\rho,r$ as in the statement of the lemma. For $i,j\in \{0,1,2,\dotsc\}$ sufficiently large with $j>i$ so that $2^{-i}\rho<1$, apply Lemma \ref{estimate_3} with $2^{-i}\rho$ in place of $\rho$ and $j-i$ in place of $i$ to get
	$$\G(a_{(r)}(x_0,2^{-j}\rho),a_{(r)}(x_0,2^{-i}\rho)) \leq C\vertiii{u}_{k,q,\lambda}\sum^{j-1}_{h=i}2^{h\left(\frac{n+rq-\lambda}{q}\right)}\cdot\rho^{\frac{\lambda-n-rq}{q}}.$$
	Thus as $\lambda>n+rq$, we see that $\sum_{h=0}^\infty 2^{h\left(\frac{n+rq-\lambda}{q}\right)}<\infty$, and thus the above shows that $(a_{(r)}(x_0,2^{-j}\rho))_j$ is a Cauchy sequence. Since $(\A_Q(\R^M),\G)$ is a complete metric space for every $M\geq 1$, this sequence therefore converges.
	
	We claim that this limit is independent of the choice of $\rho$, and so only depends on $x_0$. Indeed, for any $0<\rho_1\leq \rho_2\leq \diam(\Omega)$ we have for all sufficiently large $i$, by Lemma \ref{lemma:poly_estimate},
	\begin{align*}
		\G(a_{(r)}&(x_0,2^{-i}\rho_1),a_{(r)}(x_0,2^{-i}\rho_2))^q\\
		& \leq \frac{C_1}{(2^{-i}\rho_1)^{n+rq}}\int_{\Omega\cap B_{2^{-i}\rho_1}(x_0)}\G(P_{x_0,2^{-i}\rho_1},P_{x_0,2^{-i}\rho_2})^q\\
		& \leq \frac{C_1\cdot 2^{i(n+rq)}\cdot 2^q}{\rho_1^{n+rq}}\int_{\Omega\cap B_{2^{-i}\rho_1}(x_0)}\G(P_{x_0,2^{-i}\rho_1},u)^q + \frac{C_1\cdot 2^{i(n+rq)}\cdot 2^q}{\rho_1^{n+rq}}\int_{\Omega\cap B_{2^{-i}\rho_2}(x_0)}\G(P_{x_0,2^{-i}\rho_2},u)^q\\
		& \leq C_1\cdot 2^q\cdot\vertiii{u}^q_{k,q,\lambda}\cdot 2^{i(n+rq)}\rho^{-(n+rq)}\left((2^{-i}\rho_1)^\lambda + (2^{-i}\rho_2)^\lambda\right)\\
		& = C\vertiii{u}^q_{k,q,\lambda}2^{-i(\lambda-n-rq)}\cdot\rho_1^{-(n+rq)}\left(\rho_1^\lambda+\rho_2^\lambda\right)
	\end{align*}
	and so we see that, as $i\to\infty$ the right-hand side of this inequality $\to 0$, showing that the limit is independent of $\rho$.
	
	Now define for each $x_0\in \overline{\Omega}$ and $|p|\leq r$,
	$$v_p(x_0):= \lim_{i\to\infty}a_p(x_0,2^{-i}\rho)$$
	for any choice of $\rho\in (0,\diam(\Omega)]$. Thus $v_p:\overline{\Omega}\to \A_Q(\R^m)$. But now note that for $\rho<1$, by Lemma \ref{estimate_3} we have for every $i$ sufficiently large and $\rho<1$,
	$$\G(a_{(r)}(x_0,\rho),a_{(r)}(x_0,2^{-i}\rho))\leq C\vertiii{u}_{k,q,\lambda}\rho^{\frac{\lambda-n-rq}{q}}$$
	which follows by bounding the convergent sum in $j$ by a constant depending only on $n,r,q,\lambda$. Thus taking $i\to\infty$ we see
	$$\G(a_{(r)}(x_0,\rho),v_{(r)}(x_0))\leq C\vertiii{u}_{k,q,\lambda}\cdot\rho^{\frac{\lambda-n-rq}{q}}$$
	and this is true independently of $x_0$. In particular this shows that $\lim_{\rho\to 0}a_{(r)}(\cdot,\rho) = v_{(r)}(\cdot)$ uniformly.
\end{proof}

We now study the regularity properties of the functions $v_p$. Ultimately, we will show that $D^p u = v_p$, and thus properties for the $v_p$ will provide the necessary conclusions for $u$. The first step is to show that the \lq\lq top'' $v_p$, i.e. those with $|p|=k$, are $\alpha$-H\"older continuous in $\overline{\Omega}$ for appropriate $\lambda$. Note that from Lemma \ref{lemma:poly_estimate}, $v_p$ is independent of the choice of polynomials $P_{x_0,\rho}$ attaining the infimum.

\begin{lemma}\label{v_holder}
	Suppose $\Omega\subset \R^n$ is a convex $A$-weighted domain and $u\in \CL^{q,\lambda}_k(\Omega;\A_Q(\R^m))$ for some $\lambda>n+kq$. Then the function $v_k = (v_p)_{|p|=k}$ is H\"older continuous in $\overline{\Omega}$, with the estimate
	$$\G(v_k(x),v_k(y)) \leq C\vertiii{u}_{k,q,\lambda}\cdot|x-y|^{\frac{\lambda-n-kq}{q}}$$
	for some $C = C(n,m,k,q,Q,\lambda,A)$.
\end{lemma}

\begin{proof}
	Suppose $x,y\in \overline{\Omega}$ are such that $\rho:= |x-y|\leq \frac{1}{2}\diam(\Omega)$. Then,
	\begin{align*}
		\G(v_k(x),v_k(y)) & \leq \G(v_k(x),a_k(x,2\rho)) + \G(a_k(x,2\rho),a_k(y,2\rho)) + \G(v_k(y),a_k(y,2\rho))\\
		& \leq C\vertiii{u}_{k,q,\lambda}\cdot(2\rho)^{\frac{\lambda-n-kq}{q}} + \tilde{C}\vertiii{u}_{k,q,\lambda}\cdot\rho^{\frac{\lambda-n-kq}{q}} + C\vertiii{u}_{k,q,\lambda}\cdot (2\rho)^{\frac{\lambda-n-kq}{q}}\\
		& = C\vertiii{u}_{k,q,\lambda}\cdot |x-y|^{\frac{\lambda-n-kq}{q}}
	\end{align*}
	where in the second inequality we have used Lemma \ref{estimate_4} for the first and third terms, and Lemma \ref{estimate_2} for the second term (which is where we need the condition on $\rho$ and the fact that we can only deal with the $|p|=k$ case).
	
	If $\rho>\frac{1}{2}\diam(\Omega)$, then since $\diam(\Omega)<\infty$ and $\Omega$ is convex, the midpoint $z := (x+y)/2$ lies in $\Omega$ and obeys $|x-z|, |z-y| < \frac{1}{2}\diam(\Omega)$. Hence applying the above with $x,z$ and $z,y$, we get
	\begin{align*}
		\G(v_k(x),v_k(y)) & \leq \G(v_k(x),v_k(z)) + \G(v_k(z),v_k(y))\\
		& \leq C\vertiii{u}_{k,q,\lambda}\left(|x-z|^{\frac{\lambda-n-kq}{q}} + |z-y|^{\frac{\lambda-n-kq}{q}}\right)\\
		& = \tilde{C}\vertiii{u}_{k,q,\lambda}|x-y|^{\frac{\lambda-n-kq}{q}}
	\end{align*}
	for some $\tilde{C}$ independent of $x,y$. Hence we are done.
\end{proof}

\begin{lemma}
	Suppose $\Omega\subset\R^n$ is a convex $A$-weighted domain and $u\in \CL_k^{q,\lambda}(\Omega;\A_Q(\R^m))$, where $k\in \{1,2,\dotsc\}$ and $\lambda>n+kq$. Then for any $|p|\leq k-1$ the function $v_p$ is differentiable, and moreover for each $i\in \{1,\dotsc,Q\}$, $j\in \{1,\dotsc,n\}$,
	$$D_j v^i_p = v^i_{p+e_j}$$
	where $e_j\in \R^n$ denotes the $j$'th standard basis vector.
\end{lemma}

\begin{proof}
	We work by downwards induction on $|p|$. Fix a multi-index $p$ with $|p|\leq k-1$ and $x_0\in\overline{\Omega}$. We know from Lemma \ref{v_holder} that $v_p$, where $|p|=k$, is H\"older continuous on $\overline{\Omega}$, and so in our induction we may assume without loss of generality that $v_{p+\ell e_j}$ is continuous on $\overline{\Omega}$ for $\ell=1,2,\dotsc,k-|p|$.
	
	Set $L^i(\rho):= v_p^i(x_0) + \rho v^i_{p+e_j}(x_0)$: note this is a well-defined $Q$-valued (affine) function, since we know that $a_{(k)}(\cdot,\rho)$ converges to $v_{(k)}$ uniformly by Lemma \ref{estimate_4}, and so for each $i$ we have $v^i_{(k)} = (v^i_p)_{|p|\leq k}$ is a function $\Omega\to \R^{N}$ for some $N = N(n,m,k)$, and so given an $i$ we have a well-defined way of choosing $v^i_p$ and $v^i_{p+e_j}$. We will show that $L(\rho)$ provides the suitable linear approximation to $v_p$ at $x_0$ for the definition of differentiability.
	
	Firstly, since the $v^i$ provide a natural ordering to the unordered tuples, for each $x_0,\rho$ we can reorder $a(x,\rho)$ to assume without loss of generality that
	$$\G(a(x,\rho),v(x))^2 := \inf_\sigma\sum_i |v^i(x) - a^{\sigma(i)}(x,\rho)|^2 = \sum_i|v^i(x) - a^i(x,\rho)|^2$$
	i.e., the infimum is attained at $\sigma = \id$.
	
	Now note that for any $\sigma_1\in S_Q$ and any $\rho$ with $|\rho|$ sufficiently small:
	\begin{align*}
		\rho^{-1}\G(v_p(x_0+\rho e_j),L(\rho)) & \leq \underbrace{\rho^{-1} \G(v_p(x_0+\rho e_j),a_p(x_0+\rho e_j,2|\rho|))}_{(1)}\\
		& + \underbrace{\rho^{-1}\G(a_p(x_0+\rho e_j,2|\rho|), a_p(x_0,2|\rho|) + \rho v^{\sigma_1(\cdot)}_{p+e_j}(x_0))}_{(2)}\\
		& + \underbrace{\rho^{-1}\G(a_p(x_0,2|\rho|) + \rho v^{\sigma_1(\cdot)}_{p+e_j}(x_0), L(\rho))}_{(3)}.
	\end{align*}
	Let us look at each term individually. For (1) we simply have from Lemma \ref{estimate_4}, as $|p|\leq k-1$,
	$$(1) \leq \rho^{-1}\G(v_{(k-1)}(x_0+\rho e_j),a_{(k-1)}(x_0+\rho e_j,2|\rho|)) \leq C\vertiii{u}_{k,q,\lambda}\cdot\rho^{\frac{\lambda - n - (k-1 + 1)q}{q}}$$
	and so $(1)\to 0$ as $\rho \to 0$, since $\lambda> n+kq$. For (2), note that for any $\sigma\in S_Q$,
	\begin{align*}
		(2)^2 & \leq \rho^{-2}\sum_i\left|a^{\sigma(i)}_p(x_0+\rho e_j,2|\rho|) - a_p^{i}(x_0,2|\rho|) - \rho v_{p+e_j}^{\sigma_1(i)}(x_0)\right|^2\\
		& = \rho^{-2}\sum_i\left|D^p P^{\sigma(i)}_{x_0+\rho e_j,2|\rho|}(x)|_{x=x_0+\rho e_j} - D^p P^i_{x_0,2|\rho|}(x)|_{x=x_0} - \rho v^{\sigma_1(i)}_{p+e_j}(x_0)\right|^2.
	\end{align*}
	Now we can readily check that
	$$D^pP^{\sigma(i)}_{x_0+\rho e_j,2|\rho|}(x_0+\rho e_j) - D^p P^{\sigma(i)}_{x_0+\rho e_j,2|\rho|}(x_0) = -\sum_{\ell=1}^{|k|-p}\frac{(-1)^\ell}{\ell!}a_{p+\ell e_j}^{\sigma(i)}(x_0+\rho e_j,2|\rho|)\rho^\ell$$
	and thus substituting this in and using the triangle inequality we get:
	$$\hspace{-2em}(2)^2 \leq 2\rho^{-2}\sum_i \left|D^p [P^{\sigma(i)}_{x_0+\rho e_j,2|\rho|} - P^i_{x_0,2|\rho|}]|_{x=x_0}\right|^2 + 2\sum_i\left|v^{\sigma_1(i)}_{p+e_j}(x_0) + \sum^{k-|p|}_{\ell = 1}\frac{(-1)^\ell}{\ell !}a^{\sigma(i)}_{p+\ell e_j}(x_0+\rho e_j,2|\rho|)\rho^{\ell-1}\right|^2.$$
	Now choose $\sigma = \sigma_\rho$ such that the first sum here equals $\G(D^p P_{x_0+\rho e_j,2|\rho|}(x_0),D^p P_{x_0,2|\rho|}(x_0))^2$. Since $|x_0+\rho e_j - x_0| = \rho$, we can apply Lemma \ref{lemma:poly_estimate} and argue as in the proof of Lemma \ref{estimate_2} to see that this first term $\to 0$ as $\rho \to 0$ (again using the fact that $|p|\leq k-1$ so that we can absorb the $\rho^{-2}$ factor and still get the decay). Moreover note that, since
	$$\hspace{-1.5em}D^pP^{\sigma_\rho(i)}_{x_0+\rho e_j,2|\rho|}(x_0) - D^p P^i_{x_0,2|\rho|}(x_0) = a^{\sigma_\rho(i)}_p(x_0+\rho e_j,2|\rho|) - a^i_p(x_0,2|\rho|) + \sum^{k-|p|}_{\ell=1}\frac{(-1)^\ell}{\ell!}a^{\sigma_\rho(i)}_{p+\ell e_j}(x_0+\rho e_j,2|\rho|)\rho^\ell$$
	we see that necessarily:
	\begin{equation*}\tag{$\star$}
	|v^{\sigma_\rho(i)}_p(x_0+\rho e_j) - v^i_p(x_0)| \to 0.
	\end{equation*}
	Hence we see that
	\begin{align*}
		(2)^2 & \leq o(1) + 4\sum_i\left|v^{\sigma_1(i)}_{p+e_j}(x_0) - a^{\sigma_\rho(i)}_{p+ e_j}(x_0+\rho e_j,2|\rho|)\right|^2\\
		& \leq o(1) + 8\sum_i |v^{\sigma_1(i)}_{p+e_j}(x_0) - v^{\sigma_1(i)}_{p+e_j}(x_0+\rho e_j)|^2 + 8\sum_i |v^{\sigma_1(i)}_{p+e_j}(x_0+\rho e_j) - a_{p+e_j}^{\sigma_\rho(i)}(x_0 + \rho e_j,2|\rho|)|^2
	\end{align*}
	and so if we choose $\sigma_1 = \sigma_\rho$, we know that this $\to 0$, since then by construction the second sum equals $\G(v_{p+e_j}(x_0+\rho e_j),a_{p+e_j}(x_0+\rho e_j,2|\rho|))^2$ which $\to 0$, and the first sum $\to 0$ by the assumed continuity of $v_{p+e_j}$. Note that $\sigma_1$ was arbitrary, so we can choose it dependent on $\rho$.
	
	Finally for $(3)$ we have by any $\sigma\in S_Q$,
	\begin{align*}
		(3)^2 & \leq \rho^{-2}\sum_{i}\left|a_p^i(x_0,2|\rho|) + \rho v^{\sigma_1(i)}_{p+e_j}(x_0) - v^{\sigma(i)}_p(x_0) - \rho v^{\sigma(i)}_{p+e_j}(x_0)\right|^2\\
		& \leq 2\rho^{-2}\sum_{i}\left|a_p^i(x_0,2|\rho|) - v_p^{\sigma(i)}(x_0)\right|^2 + 2\sum_i |v^{\sigma_1(i)}_{p+e_j}(x_0) - v_{p+e_j}^{\sigma(i)}(x_0)|^2.\tag{$\dagger$}
	\end{align*}
	In particular, we can re-run this argument to see that $v_p$ is continuous at $x_0$ in the $e_j$ direction: indeed, if instead we just had $L^i \equiv v_p^i(x_0)$, and considered $\G(v_p(x_0+\rho e_j),L)$, from the above we would have shown that, for each $j$,
	$$\G(v_p(x_0+\rho e_j),L) \leq o(1) +\sum_i |a_p^i(x_0,2|\rho|)-v_p^{\sigma(i)}(x_0)|^2$$
	and so just taking $\sigma = \id$, we see this $\to 0$. Similarly we can show $v_p$ is continuous at every point.
	
	So now choosing $\sigma = \sigma_1 (= \sigma_\rho)$ in $(\dagger)$, the second sum vanishes. So we have
	$$(3)^2 \leq \rho^{-2}\sum_{i}\left|a_p^i(x_0,2|\rho|)-v_p^{\sigma_\rho(i)}(x_0)\right|^2.$$
	Now, for any sequence $\rho_t\to 0$, we can find a subsequence $\rho_{t'}$ for such $\sigma_{\rho_{t'}} \equiv \sigma'$ is constant. Hence since we know $v_p$ is continuous, from $(\star)$ we see that $v^{\sigma'(i)}_p(x_0) = v^i_p(x_0)$ for all $i$. In particular we get along this subsequence,
	$$(3)^2 \leq \rho^{-2}_{t'}\sum_{i}|a^i_{p}(x_0,2|\rho_{t'}|) - v_p^i(x_0)|^2 \leq \rho_{t'}^{-2}\G(a_{(k-1)}(x_0,2|\rho_{t'}|),v_{(k-1)}(x_0))^2$$
	which $\to 0$ by Lemma \ref{estimate_4}. Thus we see that every sequence $\rho_{t}\to 0$ has a further subsequence $(\rho_{t'})_{t'}$ such that
	$$\rho^{-1}_{t'}\G(v_p(x_0+\rho_{t'}e_j), L(\rho_{t'})) \to 0$$
	which, by elementary analysis, implies that we have $\lim_{\rho\to 0}\rho^{-1}\G(v_p(x_0+\rho e_j),L(\rho)) = 0$, i.e. $v_p$ is differentiable at $x_0$ with derivative given by $D_j v^i_p = v^i_{p+e_j}$, as required.
\end{proof}
Combining everything so far, we have now shown:
\begin{corollary}\label{reg_for_v}
	Suppose that $\Omega\subset\R^n$ is a convex $A$-weighted domain and $u\in \CL^{q,\lambda}_k(\Omega;\A_Q(\R^m))$ with $\lambda>n+kq$. Then $v_{(0)}\in C^{k,\alpha}(\overline{\Omega};\A_Q(\R^m))$, where $\alpha = \frac{\lambda - n-kq}{q}$, and moreover for all $|p|\leq k$ we have
	$$D^pv_{(0)} = v_p$$
	where $v_{(0)} = v_{(0,0,\dotsc,0)}$.
\end{corollary}

\textbf{Remark:} If for each $x_0,\rho$ the polynomial $P_{x_0,\rho}$ was of degree at most $r<k$, then we would have $a_p(x_0,\rho) \equiv 0$ for all $|p|>r$, and so $v_p\equiv 0$ for all $|p|>r$. In particular, $D^{r+1}v_{(0)} \equiv 0$, and so $v_{(0)}$ is a polynomial of degree at most $r$.

\textbf{Remark:} If $\lambda>n+(k+1)q$, then $\alpha>1$, and so $D^kv_{(0)}$ is H\"older continuous with exponent $>1$, implying that $D^k v_{(0)} \equiv \text{constant}$. Hence in this case $v_{(0)}$ is a polynomial of degree at most $k$.

Next we prove a special case of Theorem \ref{thm:campanato}, which we will use to prove the general result.

\begin{theorem}\label{initial_campanato}
	Suppose $\Omega$ is a convex $A$-weighted domain and $u\in \CL^{q,\lambda}_k(\Omega;\A_Q(\R^m))$. Suppose that $\lambda>n+kq$. Then $u\in C^{k,\alpha}(\overline{\Omega};\A_Q(\R^m))$, with the estimate
	$$[u]_{k,\alpha} \leq C\vertiii{u}_{k,q,\lambda}$$
	where $C = C(n,m,k,q,Q,A,\lambda)$.
\end{theorem}

\begin{proof}
	Fix $x_0\in \Omega$ and let $\rho>0$. Then by the triangle inequality and integrating over $\Omega\cap B_\rho(x_0)$, since $\Omega$ is $A$-weighted we get for some constant $C = C(q)$,
	\begin{align*}
		\G(a_{(0)}(x,\rho)&,u(x_0))^q \leq \underbrace{\frac{C(q)}{A\rho^n}\int_{\Omega\cap B_\rho(x_0)}\G(P_{x_0,\rho}(x),a_{(0)}(x_0,\rho))^q\ \ext x}_{(1)}\\
		& + \underbrace{\frac{C(q)}{A\rho^n}\int_{\Omega\cap B_\rho(x_0)}\G(P_{x_0,\rho}(x),u(x))^q\ \ext x}_{(2)} + \underbrace{\frac{C(q)}{A\rho^n}\int_{\Omega\cap B_\rho(x_0)}\G(u(x),u(x_0))^q\ \ext x}_{(3)}.
	\end{align*}
	For (1) notice that for all $\rho<1$,
	$$\G(P_{x_0,\rho}(x),a_{(0)}(x_0,\rho))^2 \leq \sum_i |P^i_{x_0,\rho}(x)-a^i_{(0)}(x_0,\rho)|^2 = \sum_i\left|\sum_{|p|\neq 0}\frac{a^i_p(x_0,\rho)}{p!}\cdot (x-x_0)^p\right|^2 \leq K\rho^2$$
	for some constant independent of $\rho$, since $a^i_p$ converge uniformly and so are bounded. Hence we see that $(1)\to 0$ as $\rho\to 0$. For (2) simply notice that $(2) \leq \frac{C(q)}{A\rho^n}\cdot \rho^\lambda\cdot \vertiii{u}_{k,q,\lambda}^q \to 0$ since $\lambda>n+kq$. For the final term, for $\rho$ sufficiently small we have $B_\rho(x_0)\subset \Omega$, and so $(3)\to 0$ follows for $\H^n$-a.e. $x\in \Omega$ from the Lebesgue differentiation theorem, Theorem \ref{thm:lebesgue_differentiation}. 
	
	Hence we see that for $\H^n$-a.e. $x_0\in \H^n$, $\lim_{\rho\to 0}\G(a_{(0)}(x,\rho),u(x_0))^q = 0$
	i.e., we have $v_{(0)} = u$ $\H^n$-a.e. in $\Omega$. Since $v_{(0)}\in C^{k,\alpha}(\overline{\Omega};\A_Q(\R^m))$ from Corollary \ref{reg_for_v}, this completes the proof.
\end{proof}
Finally, using Theorem \ref{initial_campanato} we are able to prove the full Campanato regularity theorem, Theorem \ref{thm:campanato}.

\begin{proof}[Proof of Theorem \ref{thm:campanato}]

Clearly since we are taking an infimum over a larger set, $\CL^{q,\lambda}_\ell(\Omega;\A_Q(\R^m))\subset \CL^{q,\lambda}_k(\Omega;\A_Q(\R^m))$ and $\vertiii{u}_{k,q,\lambda} \leq \vertiii{u}_{\ell,q,\lambda}$. Now take $u\in \CL^{q,\lambda}_k(\Omega;\A_Q(\R^m))$, where $\lambda,\ell,k$ obey the stated relations. Fix a multi-index $p$ with $\ell<|p|\leq k$. Fixing $\rho_0$, we get from the argument in Lemma \ref{estimate_3} that for any $i\geq 1$ and $x_0\in \overline{\Omega}$, since $n+|p|q\geq \lambda$,
\begin{align*}
\G(a_p(x_0,\rho_0),a_p(x_0,2^{-i}\rho_0)) & \leq C\vertiii{u}_{k,q,\lambda}\sum^{i-1}_{j=0}2^{j\left(\frac{n+|p|q-\lambda}{q}\right)}\cdot \rho_0^{\frac{\lambda-n-|p|q}{q}}\\
& = C\vertiii{u}_{k,q,\lambda}\cdot \rho_0^{\frac{\lambda-n-|p|q}{q}}\cdot\frac{2^{i\left(\frac{n+|p|q-\lambda}{q}\right)}-1}{2^{\frac{n+|p|q-\lambda}{q}}-1}\\
& \leq \tilde{C}\vertiii{u}_{k,q,\lambda}\cdot\left(\frac{\rho_0}{2^i}\right)^{\frac{\lambda-n-|p|q}{q}}
\end{align*}
where $\tilde{C}$ can depend on $|p|$, although this will not matter. Now for any $\rho< \rho_0$ we can find $i\in \{0,1,\dotsc\}$ with $\rho/\rho_0\in [2^{-(i+1)},2^{-i})$, and so by the above inequality with this $i$ we have
$$\G(a_p(x_0,\rho_0),a_p(x_0,2^{-i}\rho_0)) \leq \tilde{C}\vertiii{u}_{k,q,\lambda}\cdot \rho^{\frac{\lambda-n-|p|q}{q}}.$$
Hence we have from Lemma \ref{lemma:poly_estimate} that, estimating as in the proof of Lemma \ref{estimate_4},
$$\G(a_p(x_0,\rho),a_p(x_0,2^{-i}\rho_0))^q \leq \frac{C_1}{\rho^{n+|p|q}}\int_{\Omega\cap B_\rho(x_0)}\G(P_{x_0,\rho},P_{x_0,2^{-i}\rho_0})^q \leq C\vertiii{u}^q_{k,q,\lambda}\cdot\left(\frac{\rho^\lambda + (2^{-i}\rho_0)^\lambda}{\rho^{n+|p|q}}\right)$$
i.e.
$$\G(a_p(x_0,\rho),a_p(x_0,2^{-i}\rho_0)) \leq C\vertiii{u}_{k,q,\lambda}\rho^{\frac{\lambda - n-|p|q}{q}}.$$
Combining the last inequalities we see that for any $\rho\leq \rho_0$,
$$|a_p(x_0,\rho)| \leq \bar{C}\vertiii{u}_{k,q,\lambda}\cdot\rho^{\frac{\lambda-n-|p|q}{q}} + |a_p(x_0,\rho_0)|.$$
Now applying Lemma \ref{lemma:poly_estimate} with $F = P_{x_0,\rho_0}$ and $G \equiv 0$, we have
\begin{align*}
|a_p(x_0,\rho_0)|^q & \leq \frac{C_1}{\rho_0^{n+|p|q}}\int_{\Omega\cap B_{\rho_0}(x_0)} |P_{x_0,\rho_0}|^q\\
& \leq \frac{\tilde{C}}{\rho_0^{n+|p|q}}\left[\int_{\Omega\cap B_{\rho_0}(x_0)} \G(P_{x_0,\rho_0},u)^q + \int_\Omega |u|^q\right]\\
& \leq \frac{\tilde{C}}{\rho_0^{n+|p|q}}\left[\vertiii{u}_{k,q,\lambda}^q \cdot\rho_0^\lambda + \|u\|^q_{L^q(\Omega)}\right]\\
& = \bar{C}\|u\|^q_{k,q,\lambda}
\end{align*}
where $\bar{C}$ depends on $p$ and $\rho_0$. Note that this is $\|u\|_{k,q,\lambda}$ and not $\vertiii{u}_{k,q,\lambda}$. Hence we see that for any $\rho<1$, $x_0\in\overline{\Omega}$ and $|p|>\ell$, we have since $\lambda<n+|p|q$,
$$|a_p(x_0,\rho_0)| \leq C\|u\|_{k,q,\lambda}\cdot\rho^{\frac{\lambda-n-|p|q}{q}}.$$
Thus for all such $\rho,x_0,p$ we see
$$|a_p(x_0,\rho)| \leq C\|u\|_{k,q,\lambda}\cdot\rho^{\frac{\lambda-n-|p|q}{q}}$$
and so, for any $x_0\in \overline{\Omega}$ and $\rho<1$,
\begin{align*}
	\inf_{P\in \CP_\ell}\int_{\Omega\cap B_\rho(x_0)}\G(u,P)^q & \leq \int_{\Omega\cap B_\rho(x_0)}\G\left(u,\sum_{|p|\leq \ell}\frac{a_p(x_0,\rho)}{p!}(x-x_0)^p\right)^q\\
	& \leq 2^q\int_{\Omega\cap B_\rho(x_0)}\G(u,P_{x_0,\rho})^q + 2^q\int_{\Omega\cap B_\rho(x_0)}\G\left(P_{x_0,\rho},\sum_{|p|\leq \ell}\frac{a_p(x_0,\rho)}{p!}\cdot (x-x_0)^p\right)^q\\
	& \leq 2^q\vertiii{u}_{k,q,\lambda}^q \cdot \rho^\lambda + 2^q\int_{\Omega\cap B_\rho(x_0)}\left(\sum_i\left|\sum_{\ell<|p|\leq k}\frac{a_p(x_0,\rho)}{p!}\cdot (x-x_0)^p\right|^2\right)^{q/2}\\
	& \leq 2^q\vertiii{u}^q_{k,q,\lambda}\cdot\rho^\lambda + C\|u\|_{k,q,\lambda}^q\cdot\rho^{\lambda}\\
	& \leq \tilde{C}\|u\|^q_{k,q,\lambda}\rho^\lambda.
\end{align*}
Now when $\rho\geq 1$, this bound is immediate because we know that
$$\inf_{P\in \CP_\ell}\int_{\Omega\cap B_\rho(x_0)}\G(u,P)^q \leq \int_{\Omega\cap B_{\rho}(x_0)}\G(u,Q\llbracket 0\rrbracket)^q \leq \|u\|_{q}^q \leq \|u\|_{k,q,\lambda}^q\rho^\lambda.$$
So hence we see that
$$\vertiii{u}_{\ell,q,\lambda}^q \leq C\|u\|^q_{k,q,\lambda}$$
i.e. we see that $u\in \CL^{q,\lambda}_\ell(\Omega;\A_Q(\R^m))$. Hence as $\lambda>n+\ell q$ we can apply Theorem \ref{initial_campanato} to get the result.
\end{proof}

\subsection{Modifications relevant to minimal surface regularity theory} 

We now give two additional lemmas which are relevant for the regularity theory of stationary integral varifolds which will be the topic of Section \ref{sec:minimal_surface}. The reason for these modifications is that often we will have certain integral expressions over balls which decay with the radius at a certain rate, \textit{but only for radii such that the ball does not intersect some bad set}. One key example is a boundary point, where we may only have estimates in the interior and so cannot say anything about the boundary regularity unless we have some control of the integrals at the boundary. Another key example for us is interior branch points of multi-valued harmonic functions. Not having a uniform lower bound on the radii which we can apply the estimates to means we cannot immediately deduce that the functions lie in a certain Campanato space, and so cannot deduce any regularity near the bad set. However, if the points on the bad set also have certain integral expressions which decay, and there \textit{is} a uniform lower bound on the radii we can apply this to, we are able to deduce regularity up to the bad set so long as we can compare the integral expressions as the good points to those at the bad points. This is exactly the set up for the first case, Lemma \ref{MSC}, which we then use to prove a more general version, where there are multiple \lq\lq layers'' of bad points: these can be thought of as \lq\lq very bad'' points, \lq\lq bad points'', and \lq\lq good points'', and so forth.

\begin{lemma}\label{MSC}
	Suppose $\Omega\subset\R^n$ is a convex domain and $u:\Omega\to \A_Q(\R^m)$. Let $\Gamma\subset\overline{\Omega}$ be a closed subset and let $\Omega'\subset\Omega''\subset \overline{\Omega}$ be open such that $\overline{\Omega'}\subset\Omega''$ and $\overline{\Omega''}\subset\Omega$ are compact. Suppose also $\Omega'$ is a convex $A$-weighted domain. Suppose that there are numbers $k\in\{0,1,\dotsc\}$, $q\in [1,\infty)$, $\beta_1,\beta_2,\beta\in (0,\infty)$, $\mu\in (0,1)$, and $\epsilon \in (0,1/4)$ such that we have the following:
	\begin{enumerate}
		\item [(I)] For each $x_0\in \Gamma\cap \overline{\Omega''}$, there exists $P_{x_0}\in \CP_k$ with $\sup_{\Omega}|P_{x_0}| \leq \beta$ such that for all $0<\sigma\leq \rho/2\leq \epsilon/2$,
		$$\sigma^{-n-kq}\int_{B_\sigma(x_0)\cap \Omega}\G(u(x),P_{x_0}(x))^q\ \ext x \leq \beta_1\left(\frac{\sigma}{\rho}\right)^{q\mu}\cdot\rho^{-n-kq}\int_{B_\rho(x_0)\cap\Omega}\G(u(x),P_{x_0}(x))^q\ \ext x;$$
		\item [(II)] For each $x_0\in \overline{\Omega''}\backslash \Gamma$, there exists $P_{x_0}\in \CP_k$ such that for every $y\in \Gamma\cap\overline{\Omega''}$ and every $0<\sigma\leq \rho/2 < \frac{1}{2}\min\{1/4,\dist(x_0,\Gamma)\}$,
		$$\sigma^{-n-kq}\int_{B_\sigma(x_0)\cap\Omega}\G(u(x), P_{x_0}(x))^q\ \ext x \leq \beta_2\left(\frac{\sigma}{\rho}\right)^{q\mu}\cdot\rho^{-n-kq}\int_{B_\rho(x_0)\cap\Omega}\G(u(x),P_y(x))^q\ \ext x.$$
		Then $u\in C^{k,\lambda}(\overline{\Omega'};\A_Q(\R^m))$ for some $\lambda = \lambda(n,m,k,q,Q,A,\beta_1,\beta_2,\epsilon,\mu,\dist(\Omega',\Omega''))$ with the estimate
		$$\|u\|_{k,\lambda;\Omega'} \leq C\left(\beta^q + \int_{\Omega''}|u|^q\right)^{1/q}$$
		where $C = C(n,m,k,q,Q,A,\beta_1,\beta_2,\epsilon,\dist(\Omega',\Omega''))$.
	\end{enumerate}
\end{lemma}

\textbf{Remark:} Unlike in \cite[Lemma 4.3]{wickstable}, in (II) we only require the integral on the right-hand side to include those $P\in \CP_k$ which are of the form $P = P_y$ for some $y\in \Gamma\cap\overline{\Omega''}$. Whilst this does not modify the proof, we need this weaker hypothesis later on to choose certain $\Gamma$.

\begin{proof}
	Consider an arbitrary point $z\in \overline{\Omega'}$ and $\rho\in (0,\epsilon)$. Fix $\gamma <\frac{1}{2}\min\{1/4,\dist(\Omega',\Omega'')\}$ to be chosen later with the desired dependencies. Now if there is a point $y\in \Gamma\cap \overline{B_{\gamma\rho}(z)}$, then we have $y\in \Omega''$ and so from property (II),
	\begin{align*}
		(\gamma\rho)^{-n-kq}\int_{B_{\gamma\rho}(z)\cap\Omega}\G(u,P_y)^q & \leq 2^{n+kq}(\gamma\rho+|z-y|)^{-n-kq}\int_{B_{\gamma\rho+|z-y|}(y)\cap\Omega}\G(u,P_y)^q\\
		&  \leq 2^{n+kq}\beta_1\left(\frac{\gamma\rho+|z-y|}{\rho-|z-y|}\right)^{q\mu}\cdot(\rho - |z-y|)^{-n-kq}\int_{B_{\rho-|z-y|}(y)\cap \Omega}\G(u,P_y)^q\\
		& \leq 4^{n+kq}\beta_1\left(\frac{2\gamma}{1-\gamma}\right)^{q\mu}\cdot\rho^{-n-kq}\int_{B_\rho(y)\cap\Omega}\G(u,P_y)^q.
	\end{align*}
	So choosing $\tilde{\gamma} = \tilde{\gamma}(n,k,q,\beta_1,\mu)$ such that $4^{n+kq}\beta_1\left(\frac{2\tilde{\gamma}}{1-\tilde{\gamma}}\right)^{q\mu}<\frac{1}{4}$, we see that if $\gamma = \frac{1}{2}\min\{\tilde{\gamma},\epsilon\}$ we have
	\begin{equation}\label{E:campanato1}
		(\gamma\rho)^{-n-kq}\int_{B_{\gamma\rho}(z)\cap\Omega}\G(u,P_y)^q \leq \frac{1}{4}\cdot\rho^{-n-kq}\int_{B_\rho(y)\cap\Omega}\G(u,P_y)^q
	\end{equation}
	for any $z\in\overline{\Omega'}$ and $\rho\in (0,\epsilon)$, provided $y\in \Gamma\cap\overline{B_{\gamma\rho}(z)}$. If on the other hand $\Gamma\cap \overline{B_{\gamma\rho}(z)} = \emptyset$, then by (II) we have for any $P\in \CP_k$ of the form $P = P_y$ for some $y\in \Gamma\cap\overline{\Omega''}$ that, for any $\sigma\in (0,1/2]$,
	\begin{equation}\label{E:campanato2}
	(\sigma\gamma\rho)^{-n-kq}\int_{B_{\sigma\gamma\rho}(z)\cap\Omega}\G(u,P_z)^q \leq \beta_2\sigma^{q\mu}\cdot(\gamma\rho)^{-n-kq}\int_{B_{\gamma\rho}(z)\cap\Omega}\G(u,P)^q.
	\end{equation}
	Now fix any $z\in \overline{\Omega'}\backslash\Gamma$ with $\dist(z,\Gamma)\leq \gamma$. Let $j_*\in \{1,2,\dotsc\}$ be minimal such that $\Gamma\cap \overline{B_{\gamma^{j_*+1}}(z)} = \emptyset$. Then we can apply (\ref{E:campanato1}) with $\rho = \gamma,\gamma^2,\dotsc,\gamma^{j_*-1}$ to get that for each $j=2,\dotsc,j_*$ (if $y_*\in \Gamma\cap\overline{\Omega''}$ is chosen such that $|z-y_*| = \dist(z,\Gamma)$),
	\begin{equation}\label{E:campanato3}
		(\gamma^j)^{-n-kq}\int_{B_{\gamma^j}(z)\cap\Omega}\G(u,P_{y_*})^q \leq \frac{1}{4}\cdot (\gamma^{j-1})^{-n-kq}\int_{B_{\gamma^{j-1}}(y)\cap\Omega}\G(u,P_{y_*})^q
	\end{equation}
	and applying (\ref{E:campanato2}) with $\rho = \gamma^{j*}$ that, for each $\sigma\in (0,1/2]$ and $P = P_y$ for $y\in \Gamma\cap \overline{\Omega''}$,
	\begin{equation}\label{E:campanato4}
		(\sigma\gamma^{j_*+1})^{-n-kq}\int_{B_{\sigma\gamma^{j_*+1}}(z)\cap\Omega}\G(u,P_z)^q \leq \beta_2\sigma^{q\mu}\cdot(\gamma^{j_*+1})^{-n-kq}\int_{B_{\gamma^{j_*+1}}(z)\cap\Omega}\G(u,P)^q.
	\end{equation}
	Set $P_* := P_{y_*}$. Take $\sigma = 1/2$, $P = P_*$ in (\ref{E:campanato4}) and $j=j_*$ in (\ref{E:campanato3}) to get (after iterating (\ref{E:campanato3})):
	\begin{align*}
		&\left(\frac{1}{2}\gamma^{j_*+1}\right)^{-n-kq}\int_{B_{\frac{1}{2}\gamma^{j_*+1}}(z)\cap\Omega}\G(P_*,P_z)^q\\
		&\hspace{2em} \leq 2^q\left(\frac{1}{2}\gamma^{j_*+1}\right)^{-n-kq}\int_{B_{\frac{1}{2}\gamma^{j_*+1}}(z)\cap\Omega}\G(u,P_z)^q + 2^q\left(\gamma/2\right)^{-n-kq}\cdot (\gamma^{j_*})^{-n-kq}\int_{B_{\gamma^{j_*}}(z)\cap\Omega}\G(u,P_*)^q\\
		& \hspace{2em}\leq C\beta_2(1/2)^{q\mu}(\gamma^{j_*+1})^{-n-kq}\int_{B_{\gamma^{j_*+1}}(z)\cap\Omega}\G(u,P_*)^q + C\cdot 4^{-(j_*-1)}\cdot\gamma^{-n-kq}\int_{B_{\gamma}(y)\cap\Omega}\G(u,P_*)^q\\
		&\hspace{2em} \leq C4^{-(j_*-1)}\int_{B_\gamma(y)\cap \Omega}\G(u,P_*)^q
	\end{align*}
	where $C = C(n,k,q,\gamma,\mu,\beta_2)$. In the last inequality we have again used (\ref{E:campanato3}). Now since these are $Q$-valued polynomials of degree at most $k$, for each $j=1,2,\dotsc,j_*$ we have by substituting $\tilde{x} = f(x):=\frac{1}{2}\gamma^{j^*+1-j}(x-z)+z$ in the integral (and noting that by convexity of $\Omega$ we have $f(\Omega)\subset\Omega$):
	\begin{align*}
		(\gamma^j)^{-n-kq}&\int_{B_{\gamma^j}(z)\cap\Omega}\G(P_*(x),P_z(x))^q\ \ext x \\
		& = (\gamma^j)^{-n-kq}\int_{B_{\frac{1}{2}\gamma^{j_*+1}}(z)\cap \Omega}\G(P_*(f^{-1}(\tilde{x})),P_z(f^{-1}(\tilde{x})))^q\cdot\left(\frac{1}{2}\gamma^{j_*+1-j}\right)^{-n}\ \ext\tilde{x}\\
		& \leq (\gamma^j)^{-n-kq}\cdot\left(\frac{1}{2}\gamma^{j_*+1-j}\right)^{-n-kq}\int_{B_{\frac{1}{2}\gamma^{j_*+1}}(z)\cap\Omega}\G(P_*(\tilde{x}),P_z(\tilde{x}))^q\ \ext\tilde{x}\\
		& \leq C 4^{-(j_*-1)}\int_{B_\gamma(y)\cap\Omega}\G(u,P_*)^q\\
		& \leq C 4^{-j}\int_{B_{\gamma}(y)\cap\Omega}\G(u,P_*)^q
	\end{align*}
	for some $C = C(n,k,q,\beta_1,\beta_2,\mu,\epsilon,\dist(\Omega',\Omega''))$ independent of $j$, i.e. for each $j=1,2,\dotsc,j_*$ we have
	\begin{equation}\label{E:campanato5}
		(\gamma^j)^{-n-kq}\int_{B_{\gamma^j}(z)\cap\Omega}\G(P_*,P_z)^q \leq C4^{-j}\int_{B_\gamma(y)\cap\Omega}\G(u,P_*)^q.
	\end{equation}
	Using (\ref{E:campanato5}) with (\ref{E:campanato3}) we see that for each $j=1,\dotsc,j_*$,
	\begin{equation}\label{E:campanato6}
		(\gamma^j)^{-n-kq}\int_{B_{\gamma^j}(z)\cap \Omega}\G(u,P_z)^q \leq C4^{-(j-1)}\int_{B_\gamma(y)\cap\Omega}\G(u,P_*)^q.
	\end{equation}
	
	Now let $\rho\in (0,\gamma/2]$. Then if we have $\rho\leq \frac{1}{2}\gamma^{j_*+1}$, then we can write $\rho = \sigma\gamma^{j_*+1}$ for some $\sigma\in (0,1/2]$, and then by (\ref{E:campanato4}) we have (choosing $P = P_*$ as estimating as in the derivation of (\ref{E:campanato5})):
	\begin{align*}
	\rho^{-n-kq}\int_{B_\rho(z)\cap \Omega}\G(u,P_z)^q & \leq C\sigma^{q\mu} 4^{-(j_*-1)}\int_{B_\gamma(y)\cap\Omega}\G(u,P_*)^q\\
	& \leq C(\sigma\gamma^{j_*})^{q\mu'}\int_{B_{\gamma}(y)\cap\Omega}\G(u,P_*)^q
	\end{align*}
	where $\mu' = \min\{\mu,q^{-1}\log_{\gamma}(1/4)\}$: recall $\gamma < 1/4$ which ensures $\mu'\in (0,1)$; thus we get for such $\rho$,
	\begin{equation}\label{E:campanato7}
		\rho^{-n-kq}\int_{B_\rho(z)\cap\Omega}\G(u,P_z)^q \leq C\rho^{q\mu}\int_{B_\gamma(y)\cap\Omega}\G(u,P_*)^q.
	\end{equation}
	On the other hand, if $\rho\in (0,\gamma/2]$ has $\rho>\frac{1}{2}\gamma^{j_*+1}$, then we can find $j\in \{1,2,\dotsc,j_*\}$ for which $\gamma^{j+1}<2\rho\leq \gamma^{j}$, and so by (\ref{E:campanato6}) we have
	\begin{equation}\label{E:campanato8}
		\rho^{-n-kq}\int_{B_\rho(z)\cap\Omega}\G(u,P_z)^q \leq C\rho^{q\lambda}\int_{B_\gamma(y)\cap\Omega}\G(u,P_*)^q
	\end{equation}
	where $\lambda = \log_\gamma(1/4)\in (0,1)$. Combining (\ref{E:campanato7}) and (\ref{E:campanato8}) we see that for any $z\in\overline{\Omega'}\backslash\Gamma$ with $\dist(z,\Gamma)\leq \gamma$, we have for all $\rho\in (0,\gamma/2]$, if $\tilde{\mu} := \min\{\lambda,\mu'\}$,
	\begin{equation}\label{E:campanato9}
		\rho^{-n-kq}\int_{B_\rho(z)\cap\Omega}\G(u,P_z)^q \leq C\rho^{q\tilde{\mu}}\int_{B_\gamma(y)\cap\Omega}\G(u,P_*)^q.
	\end{equation}
		In particular since $\sup_{\Omega}|P_*|\leq \beta$ we have that
		$$\rho^{-(n+q(k+\tilde{\mu}))}\int_{B_\rho(z)\cap\Omega}\G(u,P_z)^q \leq C\left(\beta^q + \int_{\Omega''}|u|^q\right)$$
		where now $C$ also depends on $m,Q$. For $z\in \overline{\Omega'}\backslash\Gamma$ with $\dist(z,\Gamma)>\gamma$ the same inequality follows immediately from (II), for all $\rho\in (0,\gamma/2]$. If $z\in \overline{\Omega'}\cap \Gamma$, again the same inequality follows from (I). In all of this modifications to $\tilde{\mu}$, $C$, may need to be made but they have the same dependencies. Hence we see that $u\in \CL^{q,\tilde{\lambda}}_k(\Omega';\A_Q(\R^m))$ where $\tilde{\lambda} = n+q(k+\tilde{\mu})$, and so we can apply Theorem \ref{thm:campanato} to conclude that $u\in C^{k,\tilde{\lambda}}(\Omega';\A_Q(\R^m))$, with the desired estimate.
\end{proof}

The following more general version of Lemma \ref{MSC} will be necessary when there are two different types of \lq\lq bad'' points, which in our case will be boundary points and interior branch points for multi-valued harmonic functions.

\begin{lemma}\label{MSC:nonflat}
	Suppose $\Omega\subset\R^n$ is a convex domain and we have $(Q_i)_{i=1}^N$, $(m_i)_{i=1}^N\subset\Z_{\geq 1}$ and $u_i:\Omega\to \A_{Q_i}(\R^{m_i})$ for $i=1,\dotsc,N$. Suppose we have $\Omega'\subset\subset\Omega''\subset\subset\overline{\Omega}$, with $\Omega'$ a convex $A$-weighted domain. Suppose also that $\Gamma\subset\overline{\Omega}$ is a closed, non-empty subset, and we have $\Gamma_1,\dotsc,\Gamma_N\subset\overline{\Omega}\backslash \Gamma$ such that $\overline{\Gamma}_i\subset\Gamma\cup \Gamma_i$. Suppose that there are numbers $k\in \{0,1,2,\dotsc\}$, $q\in [1,\infty)$, $\beta,\beta_0\in (0,\infty)$, $(\beta_i)_{i=1}^N, (\tilde{\beta}_i)_{i=1}^N\subset(0,\infty)^N$, $\mu\in (0,1)$, $\epsilon\in (0,1/4)$ and subsets $\FP_0,\FP_i,\FP\subset\CP_k$ with $\sup_{\Omega}|P|\leq \beta$ for all $P\in \FP_0$, such that we have the following:
	\begin{enumerate}
		\item [(I)] For each $x_0\in \Gamma\cap\overline{\Omega''}$, there exists $P^1_{x_0},\dotsc,P^N_{x_0}\in \FP_0$ such that for all $0<\sigma\leq \rho/2\leq\epsilon/2$ we have:
		$$\sigma^{-n-kq}\int_{B_\sigma(x_0)\cap\Omega}\sum^N_{i=1}\G(u_i,P^i_{x_0})^q \leq \beta_0\left(\frac{\sigma}{\rho}\right)^{q\mu}\rho^{-n-kq}\int_{B_\rho(x_0)\cap\Omega}\sum^N_{i=1}\G(u_i,P_{x_0}^i)^q;$$
		\item [(II)] For each $i\in\{1,\dotsc,N\}$, we have for each $x_1\in \overline{\Omega''}\cap\Gamma_i\backslash\Gamma$ there exists $P^i_{x_1}\in \FP_i$ such that for every $P\in \FP_0$ and every $0<\sigma\leq\rho/2\leq \frac{1}{2}\min\{1/4,\dist(x_1,\Gamma)\}$ we have:
		$$\sigma^{-n-kq}\int_{B_\sigma(x_1)\cap\Omega}\G(u_i,P^i_{x_1})^q \leq \beta_i\left(\frac{\sigma}{\rho}\right)^{q\mu}\rho^{-n-kq}\int_{B_\rho(x_1)\cap\Omega}\G(u_i,P)^q;$$
		\item [(III)] For each $i\in \{1,\dotsc,N\}$, we have for each $x_2\in \overline{\Omega''}\backslash (\Gamma\cup\Gamma_i)$ there exists $P^i_{x_2}\in \FP$ such that for every $P\in \FP_0\cup\FP_i$ and every $0<\sigma\leq\rho/2\leq \frac{1}{2}\min\{1/4,\dist(x_2,\Gamma\cup\Gamma_i)\}$ we have:
		$$\sigma^{-n-kq}\int_{B_\sigma(x_2)\cap\Omega}\G(u_i,P^i_{x_2})^q \leq \tilde{\beta}_i\left(\frac{\sigma}{\rho}\right)^{q\mu}\rho^{-n-kq}\int_{B_\rho(x_2)\cap\Omega}\G(u_i,P)^q.$$
	\end{enumerate}
	Then $u_i\in C^{k,\lambda}(\overline{\Omega'};\A_{Q_i}(\R^{m_i})$ for each $i$, with the estimate
	$$\|u_i\|_{k,\lambda;\overline{\Omega'}}\leq C\left(\beta^q + \int_{\Omega''}\sum^N_{i=1}|u_i|^q\right)^{1/q}$$
	where $\lambda,C$ both depend on $n,(m_i)_i,k,q,(Q_i)_i,N,A,\beta_0,(\beta_i)_i,(\tilde{\beta}_i)_i,\epsilon,\dist(\Omega',\Omega''))$, and $\lambda$ can also depend on $\mu$.
\end{lemma}

\textbf{Remark:} The proof will follow in the same way as Lemma \ref{MSC}: essentially whenever a point $z\in \overline{\Omega'}\backslash (\Gamma\cup\Gamma_i)$ is either closer to $\Gamma$ than $\Gamma_i$, or the distances have the same order, we can complete the proof in a similar manner to Lemma \ref{MSC} using (I) and (III). The only issue comes when $\Gamma_i$ is much closer in order to $z$ than $\Gamma$, in which case we need to use (II) with (III). Then we can use (I) and (II) in the same manner as Lemma \ref{MSC} to complete the proof.

\textbf{Remark:} In our applications, we will have $\Omega = B_1(0)\cap H\subset\R^n$, where $H := \{x^1>0\}$ is a half-plane, and $\Omega' = B_{1/8}(0)\cap H$. $\Gamma$ will be a subset of $\del H$ of a priori \lq\lq bad'' points, but which we do have good integral estimates for uniformly lower bounded radii. The $\Gamma_i$ will then be the set of branch points of a multi-valued harmonic function: this \textit{will} include some boundary points which are branch points of an odd reflection of the function, which will exist so long as we have $C^1$ regularity on a neighbourhood of the boundary a priori. The condition on $\overline{\Gamma}_i$ is to ensure that $\Gamma\cup\Gamma_i$ is a closed set for each $i$, and so if $z\in \overline{\Omega'}\backslash (\Gamma\cup\Gamma_i)$, then $\dist(z,\Gamma\cup\Gamma_i)>0$: this prevents us having a point on $\del H$ in  the complement of $\Gamma\cup\Gamma_i$ which is a limit point of interior branch points.

\begin{proof}
	Note first that for each $z\in \Gamma_i\backslash \Gamma$ we may repeat the argument seen in the proof of Lemma \ref{MSC} (with only minor modifications to deal with the new form of (I)) to find that we can find $\tilde{\mu}$ such that for all $\rho>0$,
	\begin{equation}\label{E:campanato2.1}
	\rho^{-n-kq}\int_{B_\rho(z)\cap\Omega}\G(u_i,P^i_z)^q \leq C\rho^{q\tilde{\mu}}\left(\beta^q + \int_{\Omega''}\sum_i|u_i|^q\right)
	\end{equation}
	where $\tilde{\mu},C$ have all the allowed dependencies.
	
	Now fix $i\in\{1,\dotsc,N\}$. Suppose $z\in \overline{\Omega'}\backslash (\Gamma\cup\Gamma_i)$. Let $\gamma<\frac{1}{2}\min\{1/4,\dist(\Omega',\Omega'')\}$ to be chosen later and let $\rho\in (0,\epsilon)$. Note that if $\overline{B_{\gamma\rho}(z)}\cap \Gamma \neq\emptyset$, then in the same way as Lemma \ref{MSC}, choosing $y_\Gamma\in \Gamma$ with $\dist(z,\Gamma) = |z-y_\Gamma|$, we have for suitable $\gamma$,
	$$(\gamma\rho)^{-n-kq}\int_{B_{\gamma\rho}(z)\cap\Omega}\sum_j\G(u_j,P^j_{y_\Gamma})^q \leq \frac{1}{4}\cdot\rho^{-n-kq}\int_{B_\rho(y_\Gamma)\cap\Omega}\sum_j\G(u_j,P_{y_\Gamma}^j)^q.$$
	However, when $B_{\gamma\rho}(z)\cap \Gamma = \emptyset$, we cannot proceed as in the proof of Lemma \ref{MSC}, since we could have $\B_{\gamma\rho}(z)\cap\Gamma_i \neq\emptyset$, and so (III) is not applicable in the form we want, i.e. with a uniform lower bound on the radii we can apply it to.
		
	Assume now $\dist(z,\Gamma)<\gamma$: if this is not true then (i) if $\dist(z,\Gamma_i)<\gamma/4$, we can follow the proof of Lemma \ref{MSC} just with $\Gamma_i$ in place of $\Gamma$ to get the result at $z$, or (ii) If $\dist(z,\Gamma_i)\geq \gamma/4$, then we can just apply (III) to get the result at $z$. So choose $j_\Gamma$ minimal such that $\overline{B_{\gamma^{j_\Gamma+1}}(z)}\cap \Gamma = \emptyset$. Then in the same way as Lemma \ref{MSC}, for each $j=2,\dotsc,j_\Gamma$ we have (choosing $y_\Gamma$ as above and setting $P_\Gamma:= P_{y_\Gamma}$)
	\begin{equation}\label{E:campanato2.2}
	(\gamma^j)^{-n-kq}\int_{B_{\gamma^j}(z)\cap\Omega}\sum_\ell \G(u_\ell,P^\ell_{\Gamma})^q \leq \frac{1}{4}\cdot(\gamma^{j-1})^{-n-kq}\int_{B_{\gamma^{j-1}}(y_\Gamma)\cap\Omega}\sum_\ell \G(u_\ell,P^\ell_{\Gamma})^q.
	\end{equation}
	Now, if $\dist(z,\Gamma_i)\geq\gamma^{j_\Gamma+M}$, $M\geq 3$ to be chosen (dependent only on the allowed parameters), then we can still follow the proof of Lemma \ref{MSC} to get the result (as we get the same estimates on $P_i^z$ and $P_\Gamma^i$, and we have a large enough region around $z$ to apply (III)). So assume $\dist(z,\Gamma_i)<\gamma^{j_\Gamma+M}$, and let $y_*\in \Gamma_i$ be such that $|z-y_*| = \dist(z,\Gamma_i)$; note that since $z\not\in \Gamma_i\cup\Gamma = \overline{\Gamma\cup\Gamma_i}$, we have $\dist(z,\Gamma_i)>0$. So let $j_*\in\{0,1,2,\dotsc\}$ be minimal such that $\overline{B_{\gamma^{j_\Gamma+M+j_*+1}}(z)}\cap\Gamma_i = \emptyset$.
	
	We know that $\dist(y_i,\Gamma)> \gamma^{j_\Gamma+1}-\gamma^{j_\Gamma+M} \geq \gamma^{j_\Gamma+1}(1-\gamma^2)$, meaning that we can apply (II) for all $0<\sigma\leq\rho/2\leq \frac{1}{2}\gamma^{j_\Gamma+1}(1-\gamma^2)$. Thus, since $B_{\rho}(z)\subset B_{\rho+\gamma^{j_\Gamma+M}}(y_*)$, for $j=2,\dotsc,M+j_*-1$, we have from (II), in the same way as (\ref{E:campanato1}), for suitable $\gamma$,
	\begin{equation}
		(\gamma^{j_\Gamma+j+1})^{-n-kq}\int_{B_{\gamma^{j_\Gamma+j+1}}(z)\cap\Omega}\G(u_i,P^i_{y_*})^q \leq \frac{1}{4}\cdot(\gamma^{j_\Gamma+j})^{-n-kq}\int_{B_{\gamma^{j_\Gamma+j}}(y_*)\cap\Omega}\G(u_i,P^i_{y_*})^q.
	\end{equation}
	Now since $\overline{B_{\gamma^{j_\Gamma+M+j_*+1}}(z)}\cap\Gamma_i = \emptyset$, we can apply (III) to get that, for any $\sigma\in (0,1/2]$,
	$$\hspace{-2.5em}(\sigma\gamma^{j_\Gamma+M+j_*+1})^{-n-kq}\int_{B_{\sigma\gamma^{j_\Gamma+M+j_*-1}}(z)\cap\Omega}\G(u_i,P_z^i)^q \leq \tilde{\beta}_i\sigma^{q\mu}(\gamma^{j_\Gamma+M+j_*+1})^{-n-kq}\int_{B_{\gamma^{j_\Gamma+M+1+j_*}}(z)\cap\Omega}\G(u,P)^q.$$
	Now following the proof of Lemma \ref{MSC}, we get that for each $j=2,\dotsc,M+j_*-1$,
	\begin{equation*}
		(\gamma^{j_\Gamma+j+1})^{-n-kq}\int_{B_{\gamma^{j_\Gamma+j+1}}(z)\cap\Omega}\G(u_i,P^i_z)^q \leq C4^{-j}(\gamma^{j_\Gamma+2})^{-n-kq}\int_{B_{\gamma^{j_\Gamma+2}}(y_*)\cap\Omega}\G(u_i,P_{y_*}^i)^q.
	\end{equation*}
	Hence following the proof of Lemma \ref{MSC} again, we see that for any $\rho<\frac{1}{2}\gamma^{j_\Gamma+M+j_*+1}$ we have for some $\mu'\in (0,1)$,
	$$\rho^{-n-kq}\int_{B_\rho(z)\cap\Omega}\G(u_i,P^i_z)^q \leq C\rho^{q\mu'}\cdot(\gamma^{j_\Gamma})^{-n-kq-q\mu'}\int_{B_{\gamma^{j_\Gamma}}(y_*)\cap\Omega}\G(u^i,P^i_{y_*})^q.$$
	Also, if $\gamma^{j_\Gamma+j+1}<2\rho\leq \gamma^{j_\Gamma+j}$ holds for some $j=3,\dotsc,M+j_*-1$, then again we have the same inequality, e.g. if $M=4$, we see that for any $\rho<\frac{1}{2}\gamma^{j_\Gamma+3}$ we have
	$$\rho^{-n-kq-q\mu'}\int_{B_\rho(z)\cap\Omega}\G(u_i,P_z^i)^q \leq C(\gamma^{j_\Gamma})^{-n-kq-q\mu'}\int_{B_{\gamma^{j_\Gamma}}(y_*)\cap\Omega}\G(u^i,P_{y_*}^i)^q$$
	from which the desired bound (for a suitable exponent) follows from (\ref{E:campanato2.1}). 
	
	Now suppose $\rho\geq \frac{1}{2}\gamma^{j_\Gamma+3}$. Then we simply have from inclusions, since $|z-y_*| \leq \gamma^{j_\Gamma+4}<\rho$,
	$$\rho^{-n-kq-q\mu'}\int_{B_\rho(z)}\G(u_i,P_{y_*}^i)^q \leq 2^{n+kq+q\mu'}\left(\rho+|z-y_*|\right)^{-n-kq-q\mu'}\int_{B_{\rho+|z-y_*|}(y_*)\cap\Omega}\G(u_i,P_{y_*}^i)^q$$
	which then applying (\ref{E:campanato2.1}) tells us that, for every $\rho>0$,
	$$\rho^{-n-kq-q\mu'}\inf_{P\in\CP_k}\int_{B_\rho(z)}\G(u_i,P)^q \leq C\left(\beta^q+\int_{\Omega''}\sum_i|u_i|^q\right)$$
	and so since $z\in \overline{\Omega'}\backslash(\Gamma\cup\Gamma_i)$ as arbitrary, and since this inequality has already been established for $z\in \overline{\Omega'}\cap (\Gamma\cup\Gamma_i)$, we see that $u_i\in \CL^{\lambda,q}_k(\Omega';\A_{Q_i}(\R^{m_i})$, as so applying Theorem \ref{thm:campanato} we are done.
\end{proof}

\section{Applications to Minimal Submanifold Theory}\label{sec:minimal_surface}

In this section we will apply the results of Section \ref{sec:campanato} to classes of functions which arise naturally in the regularity theory of minimal submanifolds, known as \textit{blow-up classes}. We recommend the reader interested in this connection to consult the sources \cite{simoncylindrical}, \cite{wickstable}, \cite{beckerkahn}, \cite{minterwick}, \cite{minter}, and \cite{beckerwick}.

Throughout this section, we will only be interested in \textit{boundary} regularity results, and in particular will only be interested in the domain which is a $n$-dimensional half-ball, that is, $\Omega = B_1^{n+k}(0)\cap H\subset\R^{n+k}$, where $H:= \{x\in \R^{n+k}:x^1>0,\, x^{n+1}=\cdots= x^{n+k}=0\}$ is an $n$-dimensional half-plane in $\R^{n+k}$. The regularity theory of functions defined on such a half-ball arises in the situation when considering blow-ups of sequences of stationary integral varifolds converging to a fixed non-flat stationary integral cone which is itself supported on a union of half-planes meeting along a common axis (known as a \textit{classical cone} -- see \cite{minter} for a discussion in the codimension one setting).

\textbf{Remark:} When the codimension of the varifolds is $>1$, or the codimension is $1$ and the multiplicity of some half-(hyper)plane in the cone is at least $3$, it is in general not possible to guarantee that the varifolds, and thus the functions in the corresponding blow-up class, as $C^{1,\alpha}$ in the classical multi-valued sense; this is due to the existence of stationary integral varifolds which are only Lipschitz regular. There is however a slightly weaker notion that $C^{1,\alpha}$ as a multi-valued function, known as \textit{generalised}-$C^{1,\alpha}$ multi-valued function, which one can make use of in some instances in this setting (see \cite{minterwick}).

\textbf{Remark:} When one performs a blow-up procedure of a sequence of stationary integral varifolds relative to a fixed higher multiplicity (i.e. $>1$) plane as opposed to a non-flat union of half-planes, the domain of definition of the functions in the blow-up class is not the half-ball $B_1(0)\cap H$, but instead the full ball, $B_1(0)\cap \{x\in\R^{n+k}:x^{n+1}=\cdots=x^{n+k}=0\}$; such a situation arises when studying interior branch points of stationary integral varifolds. In this setting, the main difficulty is in establishing the \textit{interior} regularity of functions in the blow-up class (for instances where this is achieved, see \cite{wickstable}, \cite{minterwick}). In the setting described above, namely blow-ups relative to \textit{non-flat} cones, one typically already has an interior regularity statement for the blow-ups (provided from the planar blow-up case) and thus the primary difficulty is in establishing the boundary regularity of functions in the blow-up class; this latter question is our focus here.

\subsection{Multi-Valued Harmonic Functions}

Fix $U\subset\R^n$ an open set and $u\in C^1(U;\A_q(\R^k))$. We briefly recall some key notions for $q$-valued functions and $q$-valued harmonic functions (see also \cite{simonwick} for the case $q=2$).

\begin{defn}
	The branch set of $u$, denoted $\mathcal{B}_u$, is the set of points $x\in U$ for which there is no $\rho>0$ such that on $B_\rho(x)\subset U$, we have $u|_{B_\rho(x)} = \sum^q_{i=1}\llbracket u_i\rrbracket$ for some single-valued $C^1$ functions $u_i:B_\rho(x)\cap U\to \R^k$.
\end{defn}

\begin{defn}
	We say $u$ is \textit{harmonic} if for each $x\in U\setminus \B_u$, there is $\rho>0$ with $B_\rho(x)\subset U$ and the property that $u|_{B_\rho(x)} = \sum^q_{i=1}\llbracket u_i\rrbracket$ for some single-valued $C^1$ harmonic functions $u_i:B_\rho(x)\to \R^k$.
\end{defn}

In the case $q=2$, from \cite{simonwick} we know that if $u\in C^{1,\alpha}(U;\A_2(\R^k))$ is harmonic, then one may show:
\begin{itemize}
	\item the frequency function, $\rho\mapsto \frac{\rho\int_{B_\rho(x)}|Du|^2}{\int_{\del B_\rho(x)}|u|}$, is monotone increasing in $\rho>0$ when it is defined;
	\item $u\in C^{1,1/2}(U;\A_2(\R^k))$;
	\item the branch set $\B_u$ obeys $\dim_\H(\B_u)\leq n-2$ (in fact, one may show that $\B_u$ is countably $(n-2)$-rectifiable -- see \cite{krummelwick1}).
\end{itemize}
When $q\geq 3$, it is natural conjecture that the same result holds, except with the optimal regularity conclusion being $u\in C^{1,1/q}(U;\A_q(\R^k))$ (a key difference between these cases being that branch points where only some of the values coincide can limit onto branch points where all the follows coincide). We therefore make the following natural definition:
\begin{defn}\label{good}
	We say that a $q$-valued harmonic function $u:U\to \A_q(\R^k)$ is a \textit{good} $q$-valued harmonic function if the following conditions hold:
	\begin{itemize}
		\item $u\in C^{1,1/q}(U;\A_q(\R^k))$;
		\item $\dim_\H(\B_u)\leq n-2$;
		\item for each $x\in \B_u$ and every $0<\sigma\leq \rho<\dist(x,\del U)$, we have:
		$$\sigma^{-n}\int_{B_\sigma(x)}|u_s|^2\leq \left(\frac{\sigma}{\rho}\right)^{2(1+1/q)}\cdot\rho^{-n}\int_{B_\rho(x)}|u_s|^2$$
		where $u_s:= u-u_a$ is the symmetric part of $u$ and $u_a:= q^{-1}\sum^q_{i=1}u_i$ is its average-part.
	\end{itemize}
\end{defn}
In particular, good $q$-valued harmonic functions always have that $u_a$ is a smooth single-valued harmonic function $U\to \R^k$. As we remarked above, we already know by \cite{simonwick} that every $C^{1,\alpha}$ 2-valued harmonic function is a good $2$-valued harmonic function in the above sense (and for the purposes of the application to \cite{minterwick}, this is all that is needed).

Let us now collect a few basic properties of good $q$-valued harmonic functions which we will need later:

\begin{lemma}[Unique Continuation]\label{UC}
	Suppose $U\subset\R^n$ is open and that $u_1,u_2\in C^{1,1/q}(U;\A_q(\R^k))$ are good $q$-valued harmonic functions. Then if $u_1|_V\equiv u_2|_V$ for some open $V\subset U$, then we have $u_1\equiv u_2$ on $U$.
\end{lemma}

\begin{proof}
	By definition we know that $B:= \B_{u_1}\cup \B_{u_2}$ has $\dim_\H(B)\leq n-2$ and that $B\subset U$ is closed. In particular $U\backslash B$ is open and connected, and thus is path-connected. It suffices to show that $u_1\equiv u_2$ on $U\backslash B$ since both are continuous on $U$ and $U\cap \overline{U\backslash B} = U$. Thus suppose there is some $x\in U\backslash B$ for which $u_1(x)\neq u_2(x)$. We may then find a $x_0\in V\backslash B$ and a path $\gamma:[0,1]\to U\backslash B$ with $\gamma(0) = x_0$ and $\gamma(1) = x$. Since the image of $\gamma$ is compact, and thus closed, we may find some $r>0$ for which $B_r(\gamma([0,1]))\subset U\backslash B$. Restricting our functions now to $B_r(\gamma([0,1]))$, we can therefore reduce the lemma to proving the case where $\B_{u_1}\cup \B_{u_2} = \emptyset$.
	
	So therefore we may assume that $u_1,u_2:U\to \A_q(\R^k)$ have $u_1\equiv u_2$ on some open set and $\B_{u_1},\B_{u_2} = \emptyset$. But now the result follows immediately from the unique continuation principle for single-valued harmonic functions.
\end{proof}

We will also need a classification of $q$-valued harmonic functions which are $C^1$ and homogeneous of degree one. It is known in the case $q=2$ that every such $2$-valued harmonic function must be given by two linear functions (see \cite[Lemma 2.5]{simonwick}). Once again, it is natural to conjecture that such a result holds for general $q\geq 3$, however for now this is still an open question. We therefore make the following assumption:

\textbf{Classification Hypothesis:} \textit{Every $q$-valued symmetric harmonic function $u:\R^N\backslash\{0\}\to \A_q(\R^k)$ which is $C^1$ and homogeneous of degree one is necessarily linear (for $N\geq 2$).}

Note that we do not make any assumption on whether the harmonic function is good or not, in the sense of Definition \ref{good}; this will be necessary for our proof later. We do note that one can readily prove that the classification hypothesis holds under an inductive assumption; in particular, as this inductive assumption is true for $q=2$ (from \cite{simonwick} and \cite{krummelwick1}), we can prove the classification hypothesis for $q=3$.

\begin{lemma}
	Fix $q\in \{3,4,\dotsc\}$. Suppose that for $Q=2,3,\dotsc,q-1$, every $C^1$, homogeneous degree one, $Q$-valued harmonic function $v$ on some open subset of $\R^n$ obeys that $\B_v$ is countably $(n-2)$-rectifiable. Then, every $q$-valued symmetric harmonic function $u:\R^n\to \A_q(\R^k)$ which is $C^1$ on $\R^n\backslash\{0\}$, $n\geq 2$, and homogeneous of degree one, is necessarily comprised of $q$ single-valued linear functions.
\end{lemma}

\begin{proof}
By working with components, it suffices to prove the case when $k=1$. Set $\tilde{\CK}_u:=\{x\in \R^n:u(x) = q\llbracket 0\rrbracket \text{ and }Du(x) = q\llbracket 0\rrbracket\}\subset \CK_u$. For each point $x\in \R^n\backslash\tilde{\CK}_u$ one can find $\rho>0$ such that on $B_\rho(x)$ we can write $u$ as a sum of two multi-valued harmonic functions, both with $<q$ values. Moreover, by assumption we have that $\B_u\backslash \tilde{\CK}_u$ is countably $(n-2)$-rectifiable. In particular, it has vanishing $2$-capacity, and so for each $\epsilon>0$ we may find a cut-off function $\eta_\epsilon\in C^1_c(\R^n)$ such that $\eta_\epsilon\equiv 1$ on an $\epsilon$-neighbourhood of the multiplicity $<q$ branch set, $0\leq \eta_\epsilon\leq 1$, $\eta_\epsilon\to 0$ $\H^n$-a.e. as $\epsilon\to 0$, and $\int_{\R^n}|D\eta_\epsilon|^2\leq \epsilon^2$.

Now, following the proof of \cite[Lemma 2.5]{simonwick}, for any $\delta>0$ choose a smooth function $\gamma_\delta:\R\to \R$ which is an odd function, convex for $t\geq 0$, has $\gamma_\delta\equiv 0$ on some neighbourhood of $0$, and $\gamma_\delta'\equiv 1$ for all $t\geq \delta$. Then notice that, if we write $f_\delta:= \gamma_\delta(D_ju_i)$, then $f_\delta$ has the property that it has compact support in $S^{n-1}\backslash\tilde{\CK}_u$. In particular, by the same argument as in \cite[Lemma 2.5]{simonwick}, we have:
$$\int_{S^{n-1}}\eta_{\epsilon}\cdot|\nabla_{S^{n-1}}\gamma_\delta(D_ju_i)|^2 \leq -\int_{S^{n-1}}\eta_\epsilon\cdot\left[\Delta_{S^{n-1}}(D_j u_i)\right]\gamma_\delta(D_ju_i) + o(1) = o(1)$$
where $o(1)$ represents a term which $\to 0$ as $\epsilon\to 0$; the last equality comes from the fact that $Du$ is homogeneous of degree one. Thus, taking $\epsilon\to 0$, followed by $\delta\to 0$, we see that for each $i,j$ that $D_ju_i$ is constant; thus $u$ is given by $q$ single-valued linear functions.
\end{proof}

\subsection{Blow-Up Classes}\label{sec:blow-up}

Throughout this section, let us fix $N\in \{1,2,\dotsc\}$ and positive integers $(q_i)^N_{i=1}\subset\{1,2,\dotsc\}$; write $q:= (q_1,\dotsc,q_N)$. Let us now fix $\Omega:= B_1^{n+k}(0)\cap H$.

\begin{defn}\label{def:blow-up}
	We say that $\FB_q(\Omega)$ is a (\textit{proper}, \textit{half-plane}) \textit{blow-up class} if it obeys the following properties:
	\begin{enumerate}
	\item [($\FB1)$] Each element $v\in \FB_q(\Omega)$ is of the form $v = (v^1,\dotsc,v^N)$ with $v^i\in L^2(\Omega;\A_{q_i}(\R^k))\cap W^{1,2}_{\text{loc}}(\Omega;\A_{q_i}(\R^k))$, for each $i\in \{1,\dotsc,N\}$;
	\item [$(\FB2)$] (Interior Regularity.) If $v\in \FB_q(\Omega)$, then for $i=1,\dotsc,N$, $v^i:\Omega\to \A_{q_i}(\R^k)$ is a good $q_i$-valued harmonic function;
	\item [$(\FB3)$] (Boundary Estimates.) If $v\in \FB_q(\Omega)$ and $z\in B_1(0)\cap\del H$, then for each $\rho\in(0,\frac{3}{8}(1-|z|)]$ we have:
	$$\int_{B_{\rho/2}(z)\cap \Omega}\sum^N_{i=1}\frac{|v^i(x)-\kappa^i(z)|}{|x-z|^{n+3/2}}\ \ext x \leq C_1\rho^{-n-3/2}\int_{B_\rho(z)\cap \Omega}\sum^N_{i=1}|v^i(x)-\kappa^i(z)|^2\ \ext x$$
	where $\kappa: B_1(0)\cap \del H\to (\R^k)^N$ is a smooth, single-valued function, which obeys:
	$$\sup_{B_{5/16}(0)\cap \del H}|\kappa|^2\leq C_1\int_{B_{1/2}(0)\cap\Omega}|v|^2;$$
	\item [$(\FB4)$] (Hardt-Simon Inequality.) If $v\in \FB_q(\Omega)$ and $z\in B_1(0)\cap \del H$, then for each $\rho\in (0,\frac{3}{8}(1-|z|)]$ we have:
	$$\int_{B_{\rho/2}(z)\cap \Omega}\sum^N_{i=1}R_z^{2-n}\left(\frac{\del}{\del R_z}\left(\frac{v^i-v^i_a(z)}{R_z}\right)\right)^2 \leq C_1\rho^{-n-2}\int_{B_\rho(z)\cap \Omega}\sum^N_{i=1}|v^i-\ell_{v^i,z}|^2,$$
	where $R_z(x):= |x-z|$ and $\ell_{v^i,z}(x):= v^i_a(z) + (x-z)\cdot Dv^i_a(z)$;
	\item [($\FB5)$)] (Closure Properties.) If $v\in \FB_q(\Omega)$, then:
		\begin{enumerate}
			\item [$(\FB5\text{I})$] $v_{z,\sigma}(\cdot):=\|v(z+\sigma(\cdot))\|^{-1}_{L^2(\Omega)}v(z+\sigma(\cdot))\in \FB_q(\Omega)$ for each $z\in B_1(0)\cap \del H$ and each $\sigma\in (0,\frac{3}{8}(1-|z|)]$, whenever $v\not\equiv 0$ in $B_\sigma(z)\cap \Omega$;
			\item [$(\FB5\text{II}$)] $\|v-\ell_v\|^{-1}_{L^2(\Omega)}(v-\ell_v)\in \FB_q(\Omega)$ whenever $v-\ell_v\not\equiv 0$ in $\Omega$, where $\ell_v = (\ell_{v^1,0},\dotsc,\ell_{v^N,0})$;
		\end{enumerate}
	\item [$(\FB6)$] (Compactness Property.) If $(v_m)_m\subset\FB_q(\Omega)$, then there is a subsequence $(m')\subset (m)$ and a function $v\in \FB_q(\Omega)$ such that $v_{m'}\to v$ strongly in $L^2_{\text{loc}}(B_1(0)\cap\overline{H})$ and weakly in $W^{1,2}_{\text{loc}}(\Omega)$;
	\item [($\FB7)$] ($\epsilon$-Regularity Property.) There exists $\alpha_1 = \alpha_1(\B_q(\Omega))\in (0,1)$ and $\epsilon = \epsilon(\FB_q(\Omega))\in (0,1)$ such that the following is true: whenever $v\in \FB_q(\Omega)$ has $v^i_a(0) = 0$ and $Dv^i_a(0) = 0$ for each $i=1,\dotsc,N$, and $\|v\|_{L^2(\Omega)}=1$, and
	$$\int_{\Omega}\G(v,v_*)^2 < \epsilon$$
	for some linear function $v_* = (v^*_1,\dotsc,v^*_N)$ which is such that, for each $i$, $v^i_*:H\to \A_{q_i}(H^\perp)$ is of the form $v_*^i(x^1,\dotsc,x^n) = \sum_{j=1}^{q_i}\llbracket a^i_j x^1\rrbracket$ for some $a^i_j\in \R$ such that $\sum_j a^i_j = 0$ for each $i$, and moreover that for some $i_*\in \{1,\dotsc,N\}$ we have $q_{i_*}>1$ and $v^{i_*}_*\not\equiv 0$, then we have $\left.v\right|_{B_{1/2}(0)\cap H} \in C^{1,\alpha_1}(\overline{B_{1/2}(0)\cap H})$.
	\end{enumerate}
	Here, $C_1 = C_1(\FB_q(\Omega))\in (0,\infty)$ is a fixed constant (and so is independent of $v\in \FB_q(\Omega)$).
\end{defn}

\textbf{Remark:} A class of functions obeying these properties (or closely related properties, as one in general cannot expect $C^{1,\alpha}$ regularity in the interior, but instead $GC^{1,\alpha}$; see \cite{minterwick}) is expected to arise when performing either: (i) a coarse blow-up procedure of a sequence of certain stationary integral varifolds converging to a non-flat stationary classical cone (see \cite{minter}); or (ii) a type of fine blow-up procedure of a sequence of certain stationary integral varifolds converging to a plane of multiplicity $>1$ (see \cite{wickstable}) or a non-flat stationary classical cone where at least one half-plane in the cone has multiplicity (see \cite{minter}). In each case, as the multiplicity of each half-plane in the nearby classical cones is strictly less than the density of the vertex of the cone, one expects the above properties to hold when one has both suitable interior regularity theorems for varifolds near to lower multiplicity planes (namely, lower than the density of the cone), and when one knows that density gaps do not occur. For example, one always knows from the stationarity condition that the average-part of the blow-up relative to a plane is harmonic (see, e.g. \cite{wickstable}). A slight modification is needed in the setting of a fine blow-up procedure however, as then the fine blow-up class depends on additional parameters, and for a fixed choice of parameters the fine blow-up class does not obey the closure property $(\FB5)$ -- this will be discussed in Section \ref{sec:fine}. It should be noted that whether arbitrary blow-ups of stationary integral varifolds relative to a (multiplicity $>1$) plane obey any variational principle, such as being stationary for the Dirichlet energy for suitable variations, is a key unresolved problem in the regularity theory of stationary integral varifolds; this is closely related to the fact that one can seemingly only show \textit{weak} $W^{1,2}$ convergence to the blow-up of suitable rescaled Lipschitz approximations, as opposed to \textit{strong} $W^{1,2}$ convergence (in particular, convergence of the energy).

Now fix $q = (q_1,\dotsc,q_N)$ as above and a blow-up class $\FB_q(\Omega)$. Our main result in this setting is the following boundary regularity result for functions in $\FB_q(\Omega)$:

\begin{theorem}\label{thm:reg}
	Fix $q = (q_1,\dotsc,q_N)$ and a blow-up class $\FB_q(\Omega)$. Then, there exists $\beta = \beta(C_1,\alpha_1)\in (0,1)$, where $C_1,\alpha_1$ are the constants from Definition \ref{def:blow-up}, such that if $v\in \FB_q(\Omega)$, then in fact $v|_{B_{1/8}(0)\cap H}\in C^{1,\beta}(\overline{B_{1/8}(0)\cap H})$, with the estimate
	$$\|v\|_{1,\beta;B_{1/8}(0)\cap H}\leq C\left(\int_{B_{1/2}(0)}|v|^2\right)^{1/2}$$
	where $C = C(C_1,\alpha_1)\in (0,\infty)$. 
\end{theorem}

\textbf{Remark:} Equipped with this boundary regularity result, using the fact that the symmetric part $v^i_s$ of each component of $v\in \FB_q(\Omega)$ obeys $\left. v^i_s\right|_{\del H} = q_i\llbracket 0\rrbracket$ (this follows from $(\FB3)$), one may perform an odd reflection of $v_s$ across the boundary $\del H$, giving rise to a symmetric $C^{1,\beta}$ $q_i$-valued harmonic function on all of $B^n_1(0)$. Thus, if one knows that the interior branch set of such a function is countably $(n-2)$-rectifiable (which is known in the case $q_i=2$ by \cite{krummelwick1}), we deduce that the boundary branch set of each $v\in \FB_q(\Omega)$ is also countably $(n-2)$-rectifiable. Whether such a result passes back to the varifold level would seem to be an open question.

The proof of Theorem \ref{thm:reg} will consist of three parts. First, we are able to use $(\FB3)$ with Lemma \ref{MSC:nonflat} to deduce initial $C^{0,\alpha}(\overline{B_{1/8}(0)\cap H})$ regularity of functions in $\FB_q(\Omega)$; the exponent $\lambda = n+3/2$ present in $(\FB3)$ is however insufficient to push this to $C^1$ regularity. We will however be able to use the larger exponent present in the Hardt--Simon inequality, $(\FB4)$, to push the regularity to $C^{1,\beta}$. Indeed, the second part of the proof will involve using the Hardt--Simon inequality to classify the homogeneous degree one elements of $\FB_q(\Omega)$ by showing that they are indeed $C^1(B_1(0)\cap \overline{H}\setminus S(v))$, where $S(v)$ is the set of translation invariance of $v$; this is sufficiently high regularity to perform a reflection across $\del H$ and apply our Classification Hypothesis (after quotienting out our function by the translation invariant subset $S(v)$). Once we have achieved this classification of homogeneous degree one elements, we will be able to combine it with the Hardt--Simon inequality again to deduce that sufficiently strong integral decay estimates hold in order to apply Lemma \ref{MSC:nonflat} to deduce the result.

Before starting the proof, we first make the following observation, based on the unique continuation principle, Lemma \ref{UC}: there exists $\epsilon = \epsilon(\FB_q(\Omega))\in (0,1)$ such that if $v\in \FB_q(\Omega)$ has $v_a(0) = 0$ and $Dv_a(0) = 0$, then
\begin{equation}\label{E:lower_bound}
\epsilon\int_\Omega |v|^2 <\int_{\Omega\backslash B_{1/2}(0)}|v|^2.
\end{equation}
Indeed, if this were not true, then for each $\ell=1,2,\dotsc$, one could find $v_\ell\in \FB_q(\Omega)$ with $(v_\ell)_a(0) = 0$, $D(v_\ell)_a(0) = 0$ for which $\frac{1}{\ell}\int_{\Omega}|v_\ell|^2>\int_{\Omega\backslash B_{1/2}}|v_\ell|^2$. Setting $w_\ell:= v_\ell/\|v_\ell\|_{L^2(\Omega)}$ (note that $w_\ell\in \FB_q(\Omega)$ by $(\FB5\text{II})$), by $(\FB6)$ we may pass to a subsequence to ensure that $w_\ell\to w\in \FB_q(\Omega)$, where the convergence is strong in $L^2_{\text{loc}}(B_1(0)\cap\overline{H})$. In particular, since $\|w_\ell\|_{L^2(\Omega)}=1$, we have $\int_{B_{1/2}\backslash H}|w_\ell|>1-\ell^{-1}$ for each $\ell$, and thus $\int_{B_{1/2}\cap H}|w|^2=1$. But then we also see that $\int_K |w|^2 = 0$ for each compact $K\subset\Omega\backslash B_{1/2}(0)$, and so $w\equiv 0$ on $\Omega\backslash B_{1/2}(0)$. But from $(\FB2)$, one may invoke the unique continuation property for good $Q$-valued harmonic functions, Lemma \ref{UC}, to see that necessarily $w\equiv 0$ on $\Omega$, contradicting $\int_{B_{1/2}\cap H}|w|^2 = 1$; thus (\ref{E:lower_bound}) is established.

\subsection{Proof of Theorem \ref{thm:reg}}\label{sec:proof}

Fix $q = (q_1,\dotsc,q_N)\in \N^N$ throughout this section, and let $C_1\in (0,\infty), \alpha_1\in (0,1)$ denote the constants from Definition \ref{def:blow-up}.

\begin{lemma}[Initial $C^{0,\alpha}$ Boundary Regularity Estimate]\label{HP1}
	There exists $\alpha = \alpha(C_1,n,k)\in (0,1)$ such that if $v\in \FB_q(\Omega)$, then $v|_{B_{1/8}(0)}\in C^{0,\alpha}(\overline{B_{1/8}(0)\cap H})$; moreover, there is a $C = C(C_1,n,k)\in (0,\infty)$ such that
	$$\|v\|_{0,\alpha;B_{1/8}(0)\cap H}\leq C\left(\int_{B_{1/2}(0)\cap H}\right)^{1/2}.$$
\end{lemma}

\begin{proof}
	From $(\FB3)$ we get that, for each $z\in B_{1/8}(0)\cap \del H$ and any $0<\sigma\leq\rho/2\leq 1/32$, if we set $v^i(z):= q_i\llbracket \kappa^i(z)\rrbracket$:
	$$\sigma^{-n-3/2}\int_{B_{\sigma}(z)\cap\Omega}\sum_{i}\G(v^i,v^i(z))^2 \leq C_1\rho^{-n-3/2}\int_{B_\rho(z)\cap\Omega}\sum_i\G(v^i,v^i(z))^2.$$
	Set $\Gamma:= B_{1/8}(0)\cap \del H$ and $\Gamma_i:= \B_{v^i}$, the branch set of $v^i$. Fix $\in B_{1/8}(0)\cap H$. Then if $z\in \Gamma_i$, since for any $Q$-valued function $w$ we have $|w|^2 = |w_s|^2 + Q|w_a|^2$, and since for any good $Q$-valued function we know $w_a$ is harmonic, we thus get for all $0<\sigma\leq\rho/2<\frac{1}{2}\dist(z,\Lambda)$ and any constant $b\in \R^k$:
	\begin{align*}
		\sigma^{-n}\int_{B_\sigma(z)\cap \Omega}\G(v^i,v^i_a(z))^2 & = \sigma^{-n}\int_{B_\sigma(z)}|v^i-v^i_a(z)|^2\\
		& = \sigma^{-n}\int_{B_\sigma(z)}Q|v^i_a-v^i_a(z)|^2 + |v_s|^2\\
		& \leq \left(\frac{\sigma}{\rho}\right)^2\rho^{-n}\int_{B_\rho(z)}Q|v^i_a-b|^2 + \left(\frac{\sigma}{\rho}\right)^{2(1+1/q)}\rho^{-n}\int_{B_\rho(z)}|v^i_s|^2\\
		& = 2\left(\frac{\sigma}{\rho}\right)^{2}\rho^{-n}\int_{B_\rho(z)\cap\Omega}\G(v^i,b)^2.
	\end{align*}
	Moreover if $z\in B_{1/8}(0)\cap H\setminus (\Gamma_i\cup\Gamma)$, then on some open set disjoint from $\Gamma\cup\Gamma_i$ we can express $v^i$ as $q_i$ single-valued harmonic functions, and thus a similar inequality holds for radii $0<\sigma\leq \rho/2 <\frac{1}{2}\dist(z,\Gamma\cup\Gamma_i)$ by usual estimates for harmonic functions. Thus, we see that we can apply Lemma \ref{MSC:nonflat}, along with the assumed bounds on the boundary values from $(\FB3)$, to get $v^i\in C^{0,\alpha}(\overline{B_{1/8}(0)\cap H},\A_{q_i}(\R))$ for each $i$ and some $\alpha = \alpha(C_1,n,k)\in (0,1)$, with
	$$\|v\|_{0,\alpha;B_{1/8}(0)\cap\overline{H}}^2 \leq C\int_{B_{1/2}(0)\cap H}|v|^2$$
	as desired.
\end{proof}

\textbf{Remark:} From Lemma \ref{HP1} we see that the average-part of each $v^i$ is a harmonic function in $C^{0,\alpha}(\overline{B_{1/8}(0)\cap H})\cap C^{2,\alpha}(B_{1/8}(0)\cap H)$, which from $(\FB3)$ has smooth boundary data. As a consequence of classical elliptic boundary regularity results (see \cite{gt} or \cite{morrey}) it then follows that $v^i_\alpha\in C^{2,\alpha}(\overline{B_{1/8}(0)\cap H})$ for each $i$.

\begin{lemma}[Classification of Homogeneous Degree One Elements]\label{HP2}
	Suppose that $v\in \FB_q(\Omega)$ is homogeneous of degree one on $\Omega\backslash B_{1/4}$, i.e. $\frac{\del v}{\del R} = 0$, where $R =|x|$. Then $v$ is a linear function on $\overline{\Omega}$, and moreover the symmetric part $v^i_a$ takes the form $x\mapsto \sum_{j=1}^q\llbracket a_j^i x^1\rrbracket$ for each $i$, where $a^i_j\in \R^k$ have the property that $\sum_j a^i_j = 0$.
\end{lemma}

\begin{proof}
	The proof will combine the argument in \cite[Proposition 4.2]{wickstable} and its adaptation in \cite{minterwick} with a reflection principle in order to reduce the proof to the classification hypothesis.
	
	Fix $v\in \FB_q(\Omega)$ which is homogeneous of degree one on $\Omega\backslash B_{1/4}(0)$. First, note that necessarily $v$ is homogeneous of degree one on all of $\Omega$: indeed, the homogeneous degree one extension of each $v^i$ to all of $\Omega$ is also a good $q_i$-valued harmonic function, and thus by the unique continuation property (Lemma \ref{UC}) $v$ must necessarily coincide with this extension and so is homogeneous of degree one on $\Omega$.
	
	From the remark above, we know for each $i$ that $v^i_a$ is a $C^{2,\alpha}(\overline{B_{1/8}(0)\cap H})$ homogeneous degree one harmonic function. In particular, $Dv^i_a$ is homogeneous of degree zero and $C^{0,\alpha}(\overline{B_{1/8}(0)\cap H})$, and so is constant along rays and continuous at $0$; this implies that $Dv^i_a$ must be constant, and thus $v^i_a$ is linear for each $i$. Thus, if $v^i\equiv v_a^i$ for each $i$, we are done.
	
	So suppose $v^i\not\equiv v_a^i$ for some $i$. Applying $(\FB5\text{II})$, since here we have $\ell_v^i = v^i_a$, we may reduce to the setting where $v_a^i\equiv 0$ for each $i$ and $\|v\|_{L^2(\Omega)} = 1$. In particular, from $(\FB3)$ we must have that $\left. v^i\right|_{B_1(0)\cap \del H} \equiv q_i\llbracket 0\rrbracket$. Also, as $v\in C^{0,\alpha}(\overline{B_{1/8}(0)\cap H})$ from Lemma \ref{HP1} and since $v$ is homogeneous of degree one, we may extend $v$ to all of $\overline{H}$ by a homogeneous degree one extension, and so assume without loss of generality that $v\in C^{0,\alpha}(\overline{H})$, and that $v^i$ is a good $q_i$-valued harmonic function on $H$.
	
	Now define for each $i$:
	$$T_i(v):= \{z\in \overline{H}: v^i(z+x) = v^i(x)\text{ for all }x\in \overline{H}\}$$
	i.e. $T_i(v)$ is the set of points in $\overline{H}$ for which $v^i$ is translation invariant. Since $\left.v^i\right|_{\del H} = q_i\llbracket 0 \rrbracket$, we must have that $T_i(v)\subset \del H$ if $v^i\not\equiv 0$; by assumption this is true for some $i$. Since $v$ is homogeneous of degree one and continuous on $\overline{H}$, it is straightforward to check that $T_i(v)$ must be a subspace of $\del H$ whenever $v^i\not\equiv 0$, and so in particular in this case $\dim(T_i(v))\leq n-1$.
	
	Now set $d_i(v):= \dim(T_i(v))$, and $d(v):= (d_1(v),\dotsc,d_N(v))$. For $d = (d_1,\dotsc,d_N)$, write $\H_d$ for the set of $v\in \FB_q(\Omega)$ for which $d(v) = d$. First note that if $d_i(v)\geq n-2$ for some $i$, then $v^i$ is translation invariant along a subspace of $\del H$ of dimension at least $n-2$, meaning that $v^i$ is independent of at least $(n-2)$-coordinates (in a suitable basis). Therefore, in such coordinates, we can write $v^i(x^1,\dotsc,x^n) = \tilde{v}(x^1,x^2)$, where $\tilde{v}$ is a good $q_i$-valued harmonic function on $H_2:= \{(x^1,x^2)\in \R^2:x^1>0\}$ which is $C^{0,\alpha}$ on $\overline{H}_2$. Moreover, from the definition of a good $q_i$-valued harmonic function, we know $\dim_\H(\B_{v^i})\leq n-2$, which implies that $\B_{\tilde{v}}\cap H_2 = \emptyset$, since otherwise from the homogeneity and translation invariance we would have $\dim_\H(\B_{v^i})\geq n-1$, a contradiction.Therefore, as $H_2$ is simply connected, we could write, on \textit{all} of $H$, $\tilde{v} = \sum^{q_i}_{j=1}\llbracket \phi_j\rrbracket$, where $\phi_j:H_2\to \R$ is a smooth harmonic function. Since $\left.\phi_j\right|_{\del H_2} = 0$ for each $j$, again by the boundary regularity theory of harmonic functions and the fact that $\phi_j$ is $C^{0,\alpha}(\overline{H}_2)$, this would give that $\phi_j\in C^\infty(\overline{H}_2)$, implying as before that $D\phi_j$ was constant and thus $\phi_j$ is linear for each $j$. The zero boundary values then imply that we must have $\phi_j(x^1,x^2) = a_j x^1$ for some constants $a_j\in \R$. Summarising, we have that $v^i = \sum^{q_i}_{j=1}\llbracket a_jx^1\rrbracket$, and so in particular $d_i(v) = n-1$ (provided $v^i\not\equiv 0$). Hence we see that if $d_i(v)\geq n-2$ for each $i$, then the conclusion holds for $v$.
	
	So now suppose for contradiction that there is some $\H_d\neq \emptyset$ with $d_i<n-1$ for some $i$. Over all such $d$, let us choose one which maximises $\sum_i d_i$. If we can contradict this we prove the result.
	
	For such a $d$, fix $v\in \H_d$. Note that by the above argument we must in fact be able to find some $i_*\in \{1,\dotsc,N\}$ with $d_{i_*}(v)<n-2$. Let $K$ be any compact subspace of $\overline{H}\backslash T_{i_*}(v)$ and $\alpha\in (0,1)$. We claim that there exists $\epsilon = \epsilon(v,K,\alpha,\FB_q(\Omega))\in (0,\dist(K,T_{i_*}(v)))$ such that the following holds: for each $z\in K\cap \del H$, $\rho\in (0,\epsilon]$, and some fixed constant $C = C(\alpha,\FB_q(\Omega))$, at least one of the following holds:
	\begin{enumerate}
		\item [(a)] The conclusions of $(\FB7)$ hold on $B_{3\rho/8}(z)$; in particular, $v$ is $C^{1,\alpha}(\overline{B_{3\rho/8}(z)\cap H})$, and there is a linear function $\psi = (\psi^1,\dotsc,\psi^N)$ with $\psi^i:\overline{H}\to \A_{q_i}(\R^k)$ and $\psi|_{\del H} = 0$ such that for all $0<\tilde{\rho}\leq3\rho/8$:
		$$\tilde{\rho}^{-n-2}\int_{H\cap B_{\tilde{\rho}}(z)}\G(v(x),\psi(x))^2\ \ext x \leq C\left(\frac{\tilde{\rho}}{\rho}\right)^{2\alpha}\cdot\rho^{-n-2}\int_{H\cap B_\rho(z)}|v|^2;$$
		\item [(b)] The \textit{reverse Hardt--Simon inequality} holds, i.e.
		\begin{equation}\label{E:RHS}
			\int_{H\cap B_\rho(z)\backslash B_{\rho/2}(z)}\sum_iR_z^{2-n}\left(\frac{\del(v^i/R_z)}{\del R_z}\right)^2 \geq \epsilon\rho^{-n-2}\int_{H_\cap B_\rho(z)}|v|^2
		\end{equation}
		where $R_z:= |x-z|$.
	\end{enumerate}
	We prove this by contradiction, so suppose it were not true (with $C$ to be chosen depending only the specified parameters). Clearly if $v\equiv 0$ on $H\cap B_\rho(z)$ then there is nothing to prove, and so suppose $v\not\equiv 0$ on any such $H\cap B_\rho(z)$. Thus if the claim were false, one could find numbers $\epsilon_\ell>0$ with $\epsilon_\ell\to 0$, points $z,z_\ell\in K\cap \del H$ with $z_\ell\to z$, and radii $\rho_\ell>0$ with $\rho_\ell \leq \epsilon_\ell \to 0$ such that assertion (a) fails with $\rho = \rho_\ell$ and $z = z_\ell$ for each $\ell$, with $v|_{B_{\rho_\ell}(z_\ell)}\not\equiv 0$, and also that
	$$\int_{H\cap B_{\rho_\ell}(z_\ell)\backslash B_{\rho_\ell/2}(z_\ell)}\sum_i R_{z_\ell}^{2-n}\left(\frac{\del (v^i/R_{z_\ell})}{\del R_{z_\ell}}\right)^2 < \epsilon_\ell \rho_\ell^{-n-2}\int_{H\cap B_{\rho_\ell}(z_\ell)}|v|^2.$$
	Set $w_\ell := v_{z_\ell,\rho_\ell}$, and note that $w_\ell\in \FB_q(\Omega)$ for each $\ell$ by $(\FB5\text{I})$. By $(\FB6)$ we can pass to a subsequence to ensure that $w_\ell\to w_*\in \FB_q(\Omega)$, where the convergence is strong in $L^2_{\text{loc}}(B_1\cap\overline{H})$ and weak in $W^{1,2}_{\text{loc}}(\Omega)$. Moreover, since $v_a \equiv 0$ we have $(w_*)_a \equiv 0$. The above inequality also gives, for each $\ell$,
	$$\int_\Omega\sum_i R^{2-n}\left(\frac{\del(w^i_\ell/R)}{\del R}\right)^2 < \epsilon_\ell.$$
	It follows from this and the weak convergence in $W^{1,2}_{\text{loc}}$, which in particular gives lower semi-continuity of the energy, that $w_*$ is homogeneous of degree one on $\{x\in \Omega:\dist(x,\del\Omega)\geq \epsilon\}$ for each $\epsilon>0$; thus, $w_*$ is homogeneous of degree one on $\Omega$. Moreover, from Lemma \ref{HP1} we know $w_*$ is continuous on $B_1(0)\cap \overline{H}$ and so has a continuous homogeneous degree one extension to all of $\overline{H}$, which is a good multi-valued harmonic function from $(\FB2)$ and Lemma \ref{UC}.
	
	We claim that $w_*\not\equiv 0$. To see this, for any $Q$-valued $C^{1,\alpha}$ function $f$ defined on an interval $[a,b]\subset\R$ we have (see e.g. \cite[Equation (1.2)]{de2010almgren}):
	$$\G(f(a),f(b))\leq \int^b_a |Df|.$$
	Now fix $\w\in S^{n-1}\cap H$ and $s,r\in (1/4,1)$ with $s<r$. Applying this inequality with $f(t):= t^{-1}w^i_\ell(t\w)$ and $I= [s,r]$ for each $i$, we get
	$$\G(r^{-1}w_\ell(r\w),s^{-1}w_\ell(s\w)) \leq \int^r_s\left|\frac{\del(w_\ell(R\w)/R)}{\del R}\right|\ \ext R \leq \int^1_{1/4}\left|\frac{\del(w_\ell(R\w)/R)}{\del R}\right|\ \ext R$$
	which implies by the triangle inequality for $\G$ and the Cauchy--Schwarz inequality:
	$$|w_\ell(r\w)|^2 \leq c\left[|w_\ell(s\w)|^2 + \int^1_{1/4}\left(\frac{\del(w_\ell(R\w)/R)}{\del R}\right)\right]$$
	for some $c = c(n)$; here we have used that $r,s\geq 1/4$. Integrating this inequality over $\w\in S^{n-1}\cap H$ we get
	$$\int_{S^{n-1}\cap H}|w_\ell(r\w)|^2\ \ext\w \leq c\left[\int_{S^{n-1}\cap H}|w_\ell(s\w)|^2\ \ext \w + \int_{\Omega\backslash B_{1/4}(0)}\left(\frac{\del(w_\ell/R)}{\del R}\right)^2\right].$$
	Now, multiplying both sides by $r^{n-1}$ and integrating over $r\in (1/2,1)$, and then multiplying the resulting inequality by $s^{n-1}$ and integrating over $s\in (1/4,1/2)$ (note that $s<r$ always holds) one gets:
	$$\int_{\Omega\backslash B_{1/2}(0)}|w_\ell|^2 \leq c\left[\int_{\Omega\cap B_{1/2}\backslash B_{1/4}}|w_\ell|^2 + \int_{\Omega\backslash B_{1/4}}R^{2-n}\left(\frac{\del(w_\ell/R)}{\del R}\right)\right]$$
	(where we have used in the last integral that $R\in (1/4,1)$). Using (\ref{E:lower_bound}) we then get, for all $\ell$,
	$$\epsilon \leq c\left(\int_{\Omega\cap B_{1/2}\backslash B_{1/4}}|w_\ell|^2 + \epsilon_\ell\right)$$
	for some $\epsilon>0$ which is independent of $\ell$. Thus, as $w_\ell \to w$ strongly in $L^2(B_{1/2}\cap \overline{H})$, we see that $\epsilon\leq c\int_{\Omega\backslash B_{1/4}}|w_*|^2$; hence we see $w_*\not\equiv 0$.
	
	From the definition of $w_\ell$ it is straightforward to see that $T_i(v)\subset T_i(w_*)$ for each $i=1,\dotsc,N$. We now claim that $z\in T_i(w_*)$ for each $i$. Indeed, for each $\ell$ and $i$ write $w^i_\ell = \sum_j\llbracket f^{i,j}_\ell\rrbracket$, where $f^{i,j}_\ell :\overline{H}\to \R^k$ have we property that $(f^{i,q_i})_\ell)^{1} \leq (f^{i,q_i-1}_\ell)^{1}\leq \cdots\leq (f^{i,1}_\ell)^1$ (i.e. ordered by the first component). For notational simplicity, fix $i,j$ and write $f^{i,j}_\ell\equiv f_\ell$. Then by homogeneity of $f_\ell$, for each $y\in \overline{H}$ and $\sigma>0$ we have:
	\begin{align*}
	\sigma^{-n}\int_{B_\sigma(y)\cap H}&f_\ell(x+z)\ \ext x = \delta^{-1}_\ell \sigma^{-n}\int_{B_\sigma(y)\cap H}v^i_j(z_\ell + \rho_\ell(x+z))\ \ext x\\
	 & = (1+\rho_\ell)\delta_\ell^{-1}\sigma^{-n}\int_{B_\sigma(y)\cap H}v_j^i(z_\ell + (1+\rho_\ell)^{-1}\rho_\ell(z-z_\ell) + (1+\rho_\ell)^{-1}\rho_\ell x)\ \ext x\\
	 & = (1+\rho_\ell)^{n+1}\delta_{\ell}^{-1}\sigma^n\int_{B_{(1+\rho_\ell)^{-1}\sigma}((1+\rho_\ell)^{-1}(z-z_\ell+y))\cap H}v_j^i(z_\ell + \rho_\ell x)\ \ext x\\
	 & = (1+\rho_\ell)\cdot\left[(1+\rho_\ell)^{-1}\sigma\right]^{-n}\int_{B_{(1+\rho_\ell)^{-1}\sigma}((1+\rho_\ell)^{-1}(z-z_\ell+y))\cap H}f_\ell(x)\ \ext x
	\end{align*}
	where we have written $\delta_\ell := \|v_{z_\ell,\rho_\ell}\|_{L^2(\Omega)}$. Now letting $\ell\to \infty$, using the strong convergence in $L^2_{\text{loc}}(\overline{H})$ of $w_\ell\to w_*$, and then letting $\sigma\to 0$, we can apply the Lebesgue differentiation theorem (for single-valued functions) to get that for $\H^n$-a.e. $y\in \overline{H}$ we have $w_*(y+z) = w_*(y)$. But then from the continuity of $w_*$ provided by Lemma \ref{HP1}, this is true for every $y\in \overline{H}$; thus as $i,j$ were arbitrary, we see that $z\in T_i(w_*)$ for each $i$.\footnote{It was not strictly necessary to use integral expressions in this argument; we could have used an argument such as \cite[Lemma 3.11]{minterwick}.}
	
	In particular, we see that $z\in T_{i_*}(w_*)$, and so $d_{i_*}(w_*)>d_{i_*}(v)$. Thus we have $\sum_i d_i(w_*)>\sum_i d_i(v)$, with $w_*\not\equiv 0$; thus by the maximality of $d$, the only way to avoid contradiction is if $d_i(w_*)\geq n-1$ for each $i$, and thus $w_*$ is a combination of linear functions. Note however that as $(w_*)_a\equiv 0$, i.e. one half-plane in $w_*$ splits, we may apply the $\epsilon$-regularity property, $(\FB7)$, for all $\ell$ sufficiently large to see that in fact (a) holds for all for all $\ell$ sufficiently large; this provides the necessary contradiction to see that the dichotomy (a) -- (b) holds.
	
	Combining \ref{E:RHS} with $(\FB4\text{I})$, we then get the following dichotomy: if $z\in K\cap \del H$ and $\rho\in (0,\epsilon]$, at least one of the following holds:
	\begin{enumerate}
		\item [(i)] The conclusions of $(\FB7)$ hold on $B_{3\rho/8}(z)$; in particular, $v$ is $C^{1,\alpha}(\overline{B_{3\rho/8}(z)\cap H})$, and there is a linear function $\psi = (\psi^1,\dotsc,\psi^N)$ with $\psi^i:\overline{H}\to \A_{q_i}(\R^k)$ and $\psi|_{\del H} = 0$ such that for all $0<\tilde{\rho}\leq 3\rho/8$:
		$$\tilde{\rho}^{-n-2}\int_{H\cap B_{\tilde{\rho}}(z)}\G(v,\psi)^2 \leq C\left(\frac{\tilde{\rho}}{\rho}\right)^{2\alpha}\cdot\rho^{-n-2}\int_{H\cap B_\rho(z)}|v|^2;$$
		\item [(ii)] We have that (\ref{E:RHS}) holds and that for some $\theta = \theta(v,K,\FB_q(\Omega))$,
		$$\int_{H\cap B_{\rho/2}(z)}\sum_iR_z^{2-n}\left(\frac{\del(v^i/R_z)}{\del R_z}\right)^{2} \leq \theta\int_{H\cap B_\rho(z)}\sum_i R_z^{2-n}\left(\frac{\del(v^i/R_z)}{\del R_z}\right)^2.$$
	\end{enumerate}
	We claim that for this, at least one of the following holds for each $z\in K\cap \del H$:
	\begin{enumerate}
		\item [(I)] The conclusions of $(\FB7)$ hold on some neighbourhood of $z$, and moreover there is a linear function $\psi_z = (\psi^1,\dotsc,\psi^N_z)$ with $\psi^i_z:\overline{H}\to \A_{q_i}(\R^k)$ and $\psi_z|_{\del H} = 0$ such that for all $\rho\in (0,3\epsilon/8]$:
		$$\rho^{-n-2}\int_{H\cap B_\rho(z)}\G(v,\psi)^2\leq C\rho^{2\mu}\int_{H\cap B_\epsilon(z)}|v|^2$$
		for some $C = C(v,K,\FB_q(\Omega))$;
		\item [(II)] We have that (\ref{E:RHS}) holds with $\rho = 2^{-i}\epsilon$ for each $i\in \{1,2,\dotsc\}$, and hence
		$$\int_{H\cap B_\sigma(z)}\sum_i R^{2-n}_z\left(\frac{\del(v^i/R_z)}{\del R_z}\right)^2 \leq \beta\left(\frac{\sigma}{\rho}\right)^{2\mu}\int_{H\cap B_\rho(z)}\sum_i R^{2-n}_z\left(\frac{\del(v^i/R_z)}{\del R_z}\right)^2$$
		for all $0<\sigma\leq \rho/2\leq \epsilon/2$;
	\end{enumerate}
	here, $\mu = \mu(v,K,\FB_q(\Omega))\in (0,1)$ and $\beta = \beta(v,K,\FB_q(\Omega))\in (0,\infty)$. Indeed, for each fixed $z\in K\cap \del H$ the dichotomy (i) or (ii) above holds for $\rho = 2^{-i}\epsilon$, $i=0,1,2,\dotsc$; let $I$ be the first time that (i) holds. If $I=0$, the have alternative (I); also, if $I\geq 1$, then iterating (ii) for $i=0,1,\dotsc,I-1$ and combining the estimate provided with (i) as $i=I$ as well as $(\FB4\text{I})$ and (\ref{E:RHS}) gives again (I). If no such $I$ exists, i.e. if (ii) holds for each such $i$, then iterating (ii) for all $i$ (and interpolating between the scales in the usual fashion) we get (II).
	
	Finally, in the case where (II) holds, we can again use $(\FB4\text{I})$ and (\ref{E:RHS}) in conjunction with the estimate in (II) to replace (II) with:
	\begin{enumerate}
		\item [(II)'] For all $0<\sigma\leq\rho/2\leq \epsilon/4$ we have (writing $|v|^2 = \sum_i|v^i|^2$):
		\begin{equation}
			\sigma^{-n-2}\int_{H\cap B_\sigma(z)}|v|\leq \beta\left(\frac{\sigma}{\rho}\right)^{2\mu}\rho^{-n-2}\int_{H\cap B_\rho(z)}|v|^2.
		\end{equation}
	\end{enumerate}
	Now let $\Gamma$ be the set of $z\in K\cap \del H$ at which (II) holds; as the points at which (I) hold form an open subset of $\del H$, we know that $\Gamma$ is a closed subset of $\del H$. At any point $z\in \del H\backslash \Gamma$, we know that (I) holds, and thus $v$ is $C^{1,\alpha}$ up to the boundary on some neighbourhood of $z$. In particular, since $v|_{\del H}\equiv 0$ we are able to apply an odd reflection across $\del H$ on this neighbourhood to see that $v$ locally extends about $z$ to a $C^{1,\alpha}$ good multi-valued harmonic function on some ball $B_{\rho_z}(z)$; in particular, the branch set of this extension must also be closed and have dimension at most $n-2$. Now define $\Gamma_i$ to be the union of the interior branch set (in $H$ of $v^i$ along with any branch point on $\del H$ which arise in such a reflection process at points $z\in \del H\backslash \Gamma$. It is simple to see that $\overline{\Gamma}_i\subset \Gamma\cup\Gamma_i$; indeed, away from $\del H$ the interior branch set is closed, and so if $z\cap\overline{\Gamma}\cap \del H$, we either have $z\in \Gamma$ or if not then $z$ is still the limit of \textit{interior} branch points of a good multi-valued harmonic function with $z$ in the interior of its domain of definition.
	
	Therefore, whenever $K$ is convex (and $A$-weighted for some $A>0$; we can choose this dependent only on $n,k$), we are able to apply Lemma \ref{MSC:nonflat} (remembering that those points in $\del H\backslash(\Gamma\cup\Gamma_i)$ have a local good multi-valued harmonic function decomposition locally) to get that $v$ is $C^{1,\tilde{\alpha}}$ on $K$ for some $\tilde{\alpha} = \tilde{\alpha}(v,K,\FB_q(\Omega))$. Therefore choosing $K$ to be half-balls which exhaust $\overline{H}\backslash T_{i_*}(v))$, we deduce that $v\in C^1(\overline{H}\backslash T_{i_*}(v))$ (note that we are not able to deduce $C^{1,\alpha_*}$ regularity for some fixed $\alpha_*$ as we could have $\tilde{\alpha}\to 0$ for some sequence of domains $K$).
	
	Now consider $v^{i_*}$: this is translation invariant along $T_{i_*}(v)$, and so we may write (after possibility rotating) $v^{i_*}(x^1,\dotsc,x^n):= f(x^1,\dotsc,x^m)$ for some $f:H_m\to \A_{q_{i_*}}(\R^k)$ which is a good $q_{i_*}$-valued harmonic function; here $H_m:= \{(x^1,\dotsc,x^m)\in \R^m: x^1>0\}$, where $m := n - d_{i_*}(v)\geq 3$. From the above we know that $f\in C^1(\overline{H}_m\backslash\{0\})$ with $f|_{\del H_m} = q_{i_*}\llbracket 0\rrbracket$.
	
	We may now define an odd reflection $F:\R^m\to \A_{q_{i_*}}(\R^k)$ of $f$ by:
	$$F(x^1,\dotsc,x^m) := \begin{cases}
		\hspace{0.85em}f(x^1,x^2,\dotsc,x^m) & \text{if }x^1\geq 0;\\
		-f(-x^1,x^2,\dotsc,x^m) & \text{if }x^1<0.
	\end{cases}$$
	Now since we have $f\in C^1(\overline{H}_m\backslash\{0\})$, we see that $F\in C^1(\R^m\backslash\{0\})$. Moreover, $F$ is a $q_{i_*}$-valued harmonic function: this is clear on $\R^m\backslash\{x^1=0\}$, and on $\{x^1=0\}$, if we have a point $z\in \{x^1=0\}\backslash \CK_F$, where $\CK_F:= \{x:F^i(x) = F^j(x)\text{ and }DF^i(x) = DF^j(x)\text{ for some }i\neq j\}$, then locally about $x$ we necessarily have that $f$ will be given by $q_{i_*}$ single-valued harmonic functions, and so the usual reflection principle for single-valued harmonic functions ensures that $F$ is given by $q_{i_*}$ single-valued harmonic functions on the neighbourhood extended by the reflection.
	
	Thus, $F$ is a homogeneous degree one $C^1$, $q_{i_*}$-valued harmonic function on $\R^m\backslash\{0\}$ (note that we do not know whether $F$ is a \textit{good} $q_{i_*}$-valued harmonic function as it could have an $\H^{m-1}$-positive set of branch points on $\{x^1=0\}$). But the Classification Hypothesis enables us to conclude that $F$ must be linear. However, since $F|_{\{x^1=0\}} = q_{i_*}\llbracket 0\rrbracket$, we must have that $F$ takes the form $x\mapsto \sum_j\llbracket a_jx^1\rrbracket$ for some $a_j$; this implies that $v^{i_*}(x) = \sum_j \llbracket a_j x^1\rrbracket$, which means that $d_{i_*}(v) = n-1$; this provides the contradiction necessary to show that $\H_d = \emptyset$ unless $d = (n-1,\dotsc,n-1)$, completing the proof.
\end{proof}

\textbf{Remark:} It should be stressed that the crucial difference between the multi-valued case and the single-valued case is the ability to reflect. The reflection principle for single-valued harmonic functions requires no assumption on the boundary regularity of the derivative, whilst in the multi-valued setting we first need to establish suitable $C^1$ regularity at the boundary before we can reflect to get a $C^1$ function on the whole plane.

With the characterisation provided in Lemma \ref{HP2}, we can now prove Theorem \ref{thm:reg}.

\begin{proof}[Proof of Theorem \ref{thm:reg}]
	The proof essentially follows the same lines as that of Lemma \ref{HP2}, using the classification of homogeneous degree one elements in $\FB_q(\Omega)$ to establish a suitable reverse Hardt--Simon inequality for arbitrary $v\in \FB_q(\Omega)$.
	
	Indeed, we claim that there is $\epsilon = \epsilon(\FB_q(\Omega))$ such that the following holds: if $v\in \FB_q(\Omega)$, then for all $z\in B_1\cap\del H$ and $\rho\in (0,1/32]$ at least one of the following must hold:
	\begin{enumerate}
		\item [(a)] The conclusions of $(\FB7)$ hold on $B_{3/256}(z)$; in particular, $v$ is $C^{1,\alpha}(B_{3/256}(z)\cap\overline{H})$, and there is a linear function $\psi = (\psi^1,\dotsc,\psi^N)$ with $\psi^i:\overline{H}\to \A_{q_i}(\R^k)$ and $\psi|_{\del H} = 0$ such that for all $0<\rho\leq 3/256$:
		$$\rho^{-n-2}\int_{H\cap B_\rho(z)}\G(v - \ell_v,\psi)^2 \leq C\rho^{2\alpha}\cdot\int_{H\cap B_{1/32}(z)}|v-\ell_v|$$
		where $\ell_v = (\ell_v^1,\dotsc,\ell_v^N)$ ($\ell_v^N \equiv \ell_{v^N_a,z}$ is the first-order linear approximation to $v^N_a$ at $z$);
		\item [(b)] The reverse Hardt--Simon inequality holds:
		$$\int_{\Omega\cap B_\rho(z)\backslash B_{\rho/2}(z)}\sum_i R_z^{2-n}\left(\frac{\del((v^i-v^i_a(z))/R_z)}{\del R_z}\right)^{2} \geq\epsilon\rho^{-n-2}\int_{\Omega\cap B_\rho(z)}\sum^i|v^i-\ell_{v^i_a,z}|^2.$$
	\end{enumerate}
		Indeed, setting $w:= v_{z,\rho}$ and $\tilde{w}:= \|w-\ell_w\|^{-1}_{L^2(\Omega)}(w-\ell_w)$, we know that $w,\tilde{w}\in \FB_q(\Omega)$ by $(\FB5\text{I})$ and $(\FB5\text{II})$, respectively, and so it suffices to consider the case $\rho=1$, $z=0$, and $v\in \FB_q(\Omega)$ with $v_a(0) = 0$, $Dv_a(0) =0$, and $\|v\|_{L^2(\Omega)} =1$. If this were not true, then one could find a sequence $\epsilon_\ell\downarrow 0$ and $(v_\ell)_\ell$ of such $v$ for which both (a) and (b) fail with $\epsilon_\ell$ in place of $\epsilon$ and $v_\ell$ in place of $v$ for each $\ell$. In particular,
		$$\int_\Omega \sum_i R^{2-n}\left(\frac{\del (v^i/R)}{\del R}\right)^2 < \epsilon_\ell.$$
		By $(\FB6)$ we may pass to a subsequence to find a limit $v_\ell\to v_*\in \FB_q(\Omega)$; arguing the same way as in Lemma \ref{HP2}, we see that $v_*$ must be homogeneous of degree one, $v_*\not\equiv 0$, and $(v_*)_a(0) = 0$ and $D(v_*)_a(0) = 0$. But then Lemma \ref{HP2} gives that $v_*$ must be linear, at which point $(v_*)_a(0) = 0$ and $D(v_*)_a(0) = 0$ imply that $(v_*)_a\equiv 0$. Hence one may apply $(\FB7)$ for all $\ell$ sufficiently large to get that $v_\ell$ obeys (a), a contradiction.
		
		From the dichotomy (a) -- (b), one may then argue in the same way as Lemma \ref{HP2} to deduce the dichotomy (I) -- (II), which can then be used with Lemma \ref{MSC:nonflat} to deduce the result.
\end{proof}

\subsection{Modifications to Fine Blow-Up Classes}\label{sec:fine}

The rest of the paper will be devoted to discussing some modifications of the regularity theorem seen in Section \ref{sec:blow-up} to other setting which naturally arise in the regularity theory of stationary integral varifolds, but which do not quite fit the profile seen in Section \ref{sec:blow-up}. In Section \ref{sec:blow-up}, the regularity theory is for classes of functions arising in the simplest situation, namely for certain sequences of varifolds converging to a fixed stationary integral cone supported on half-planes meeting along an axis. In recent years, a different blow-up procedure known as a \textit{fine blow-up} has been used to analyse degenerate situations (see \cite{wickstable}, \cite{minter}). The difference in this situation is that the construction of a fine blow-up class depends on an additional parameter, $M\in (1,\infty)$, and the class of functions (which are always defined on half-balls) constructed through a fine blow-up procedure is not closed under the operations in $(\FB5)$ as this parameter $M$ could change. In the applications thus far however it can be shown that the functions in $(\FB5)$ belong to another fine blow-up class for which the parameter $M$ has only changed by a fixed constant. All other properties $(\FB1) - (\FB7)$ remain unchanged, and under these observations a close examination of the proof in Section \ref{sec:proof} reveals that the same proof will work, subject to the functions of interest perhaps switching to another, \textit{fixed} blow-up class. To be more precise, let us make the following definition; here, as usual we fix integers $q_1,\dotsc,q_N\in \{1,2,\dotsc\}$ and write $q = (q_1,\dotsc,q_N)$.

\begin{defn}
	We say that $(\FB^F_{q;\,M}(\Omega))_{M\in (1,\infty)}$ is a (\text{proper}) \textit{fine blow-up family} if for each $M\in (1,\infty)$, the class $\FB^F_{q;\,M}(\Omega)$ obeys all properties of a blow-up class seen in Definition \ref{def:blow-up}, with constant $C$ independent of $M$, except for $(\FB5)$, where instead there is some $M_* = M_*(M)$ such that if $v\in \FB^F_{q;\,M}(\Omega)$ and $v_*$ is one of the functions in $(\FB5)$ generated from $v$, we have $v_*\in \FB^{F}_{q;\,M_*(M)}(\Omega)$. We call each $\FB^F_{q;\,M}(\Omega)$ a \textit{fine blow-up class}.
\end{defn}

\textbf{Remark:} In \cite{wickstable} and \cite{minter}, one may find $M_0>1$ independent of $M$ for which $M_*(M):= M_0M$.

As noted above, one may follow the proof in Section \ref{sec:proof} to prove the following regularity theorem for functions in a fine blow-up class:

\begin{theorem}\label{thm:fine_reg}
	If $\FB^F_{q;\,M}(\Omega)$ is a fine blow-up class, then if $v\in \FB^F_{q;\,M}(\Omega)$ we have $v|_{B_{1/8}(0)}\in C^{1,\alpha}(\overline{B_{1/8}(0)\cap H})$ for some $\alpha = \alpha\in (0,1)$ which is independent of $M$, with the estimate
	$$\|v\|_{1,\alpha;B_{1/8}(0)\cap \overline{H}}\leq C\left(\int_{B_{1/2}(0)\cap H}|v|^2\right)^{1/2}$$
	where $C \in (0,\infty)$ is independent of $M$.
\end{theorem}

Once again, given appropriate regularity theorems for $C^{1,\alpha}$ multi-valued functions, one may prove that the boundary branch set of $v\in \FB^F_{q;\,M}(\Omega)$ is countably $(n-2)$-rectifiable.

\bibliographystyle{alpha} 
\bibliography{references}

\end{document}